\documentclass[12pt,reqno]{amsart}
\usepackage{amsmath,amsthm,amsfonts,amssymb,times}
\usepackage[mathscr]{eucal}
\usepackage{verbatim}
\usepackage{url}
\setbox0=\hbox{$+$}
\newdimen\plusheight
\plusheight=\ht0
\def\+{\;\lower\plusheight\hbox{$+$}\;}

\setbox0=\hbox{$-$}
\newdimen\minusheight
\minusheight=\ht0
\def\-{\;\lower\minusheight\hbox{$-$}\;}

\setbox0=\hbox{$\cdots$}
\newdimen\cdotsheight
\cdotsheight=\plusheight
\def\cds{\lower\cdotsheight\hbox{$\cdots$}}

\makeatletter
\def\leqalignno#1{\displ@y \tabskip\z@ plus\@ne fil
  \halign to\displaywidth{\hfil$\@lign\displaystyle{##}$\tabskip\z@skip
    &$\@lign\displaystyle{{}##}$\hfil\tabskip\z@ plus\@ne fil
    &\kern-\displaywidth\rlap{$\@lign\hbox{\rm##}$}\tabskip\displaywidth\crcr
    #1\crcr}}
\makeatother
\let\dotlessi=\i

\newcommand{\eb}{\begin{equation}}
\newcommand{\ee}{\end{equation}}

\newcommand{\df}{\dfrac}
\newcommand{\tf}{\tfrac}

\renewcommand{\Re}{\operatorname{Re}}

\newcommand{\Reg}{\operatorname{Reg}}

\renewcommand{\b}{\beta}

\renewcommand{\k}{\kappa}

\renewcommand{\Re}{\text{Re}}

\renewcommand{\(}{\left\(}
\renewcommand{\)}{\right\)}
\renewcommand{\[}{\left\[}
\renewcommand{\]}{\right\]}
\renewcommand{\i}{\infty}
\renewcommand{\pmod}[1]{\,(\textup{mod}\,#1)}
\numberwithin{equation}{section}
 \theoremstyle{plain}
\newtheorem{theorem}{Theorem}[section]

\newtheorem{corollary}[theorem]{Corollary}

\newtheorem{definition}[theorem]{Definition}

\DeclareMathOperator{\Realp}{Re}
\allowdisplaybreaks

\begin{document}
\title[ Divisor Sums in Algebraic Number Fields]
{Two-parameter Identities for Divisor Sums in Algebraic Number Fields}
\author{Bruce C.~Berndt, Martino Fassina, Sun Kim, and Alexandru Zaharescu}
\address{Department of Mathematics, University of Illinois, 1409 West Green
Street, Urbana, IL 61801, USA} \email{berndt@illinois.edu}
\address{Fakult\"at f\"ur Mathematik, Universit\"at Wien, Oskar-Morgenstern-Platz 1, 1090 Wien, Austria}
\email{martino.fassina@univie.ac.at}
\address{Department of Mathematics, and Institute of Pure and Applied Mathematics, Jeonbuk National University, 567 Baekje-daero, Jeonju-si, Jeollabuk-do 54896, Republic of Korea}
\email{sunkim@jbnu.ac.kr}
\address{Department of Mathematics, University of Illinois, 1409 West Green
Street, Urbana, IL 61801, USA; Institute of Mathematics of the Romanian
Academy, P.O.~Box 1-764, Bucharest RO-70700, Romania}
\email{zaharesc@illinois.edu}


\begin{abstract} In a one-page fragment published with his lost notebook, Ramanujan stated two double series identities associated, respectively, with the famous \emph{Gauss Circle}  and \emph{Dirichlet Divisor} problems.  The identities contain an "extra'' parameter, and it is possible that Ramanujan derived these identities with the intent of attacking these famous problems. Similar famous unsolved problems are connected with $f_K(n)$, the number of integral ideals of norm $n$ in an algebraic number field $K$. In this paper we establish Riesz sum identities containing an ``extra'' parameter and involving   $f_K(n)$, or divisor functions associated with $K$.  Upper bounds for the sums as the upper index tends to infinity are also established.  
\end{abstract}



\maketitle

\section{Introduction}  Let $d(n)$ denote the number of positive divisors of the positive integer $n$, and set
 \begin{equation}\label{Dofx}
D(x):= {\sum_{n\leq x}}^{\prime}d(n),
\end{equation}
where the prime on the summation sign indicates that if $x$ is an integer, then only $\tf12 d(x)$ is counted in \eqref{Dofx}.
Dirichlet showed that \cite{dirichlet}
\begin{equation}
D(x)=x(\log{x}+2\gamma-1)+\df14+\Delta(x),
\end{equation}
where $\gamma$ denotes Euler's constant, and $\Delta(x)$ is the
``error term.''
One of the most famous unsolved problems in analytic number theory is to find the optimal order of magnitude for $\Delta(x)$, as $x \to \i$; this is the \emph{Dirichlet Divisor Problem}. (The fraction $\frac14$ is present because it arises naturally from analytic investigations and, of course, does not affect nontrivial  bounds for $\Delta(x)$.) Dirichlet's elementary argument  also yields the first upper bound for $\Delta(x)$, namely,
\begin{equation}
\Delta(x)=O(\sqrt{x}),\qquad x\to\i .
\end{equation}
Berndt, Kim, and Zaharescu \cite{monthly} provide a summary of upper bounds that have been achieved for $\Delta(x)$, as $x\to\infty$, in the twentieth century. 
A theorem of G.~H.~Hardy \cite{hardy1916} implies that $\Delta(x)\neq O(x^{1/4})$, as $x\to \infty$.

For over a century,  those seeking increasingly  better upper bounds for $\Delta(x)$ have found that 
 a famous formula of G.~F.~Vorono\"{\dotlessi} \cite{voronoi} to be useful.  In order to state his formula, we need some definitions.  Let $J_{\nu}(x)$ denote the ordinary Bessel function of order $\nu$.
 The Bessel function $Y_{\nu}(z)$
 of the second kind  \cite[p.~64, Equation (1)]{watson} is defined by
\begin{equation}\label{b.Y}
Y_{\nu}(z):=\df{J_{\nu}(z)\cos(\nu\pi)-J_{-\nu}(z)}{\sin(\nu\pi)},
\end{equation}
and the modified Bessel function $K_{\nu}(z)$
\cite[p.~78, Equation (6)]{watson} is defined,
for $-\pi<\arg z<\tf12\pi$, by
\begin{equation}\label{b.K}
K_{\nu}(z):=\df{\pi}{2}\,\df{e^{\pi{i}\nu/2}J_{-\nu}(iz)-e^{-\pi{i}\nu/2}J_{\nu}(iz)}{\sin(\nu\pi)}.
\end{equation}
If $\nu$ is an integer $n$, it is understood that we take the limits in \eqref{b.Y} and \eqref{b.K} as $\nu\to n$.  Then Vorono\"{\dotlessi}'s formula is given by
\begin{equation}\label{b.vor}
{\sum_{n\leq x}}^\prime d(n) =x\left(\log x
+2\gamma-1\right)+\frac{1}{4}+\sum_{n=1}^{\i}d(n)\left(\df{x}{n}\right)^{1/2}I_1(4\pi\sqrt{nx}),
\end{equation}
where $x>0$, and $I_1(z)$ is defined by
\begin{equation}\label{b.c.9}
I_{\nu}(z) := -Y_{\nu}(z)-\df{2}{\pi}K_{\nu}(z).
\end{equation}

In a one-page fragment published with his lost notebook \cite[p.~335]{lnb}, Ramanujan offered without proof an identity involving  the same Bessel functions that appear on the right-hand side of \eqref{b.vor}.
 For $x>0$ and $0<\theta<1$,
\begin{align}\label{b.c.11}
{\sum_{n\leq x}}^{\prime}&\left[\frac{x}{n}\right]\cos(2\pi n\theta)=
\df{1}{4} -x\log(2\sin(\pi\theta)) \notag\\
&+\df{1}{2}\sqrt{x}\sum_{m=1}^{\i}\sum_{n=0}^{\i}
\left\{\df{I_1\left(4\pi\sqrt{m(n+\theta)x}\right)}{\sqrt{m(n+\theta)}}
+\df{I_1\left(4\pi\sqrt{m(n+1-\theta)x}\right)}{\sqrt{m(n+1-\theta)}}\right\},
\end{align}
where $I_1(z)$ is defined in \eqref{b.c.9}.  An elementary argument shows that
\begin{equation*}
D(x)={\sum_{n\leq x}}^{\prime}\left[\df{x}{n}\right].
\end{equation*}
Hence, if $\theta=0$, the left side of \eqref{b.c.11} reduces to $D(x)$. Thus, \eqref{b.c.11} can be considered as a two-variable analogue of \eqref{b.vor}. In his first letter to Hardy \cite[p.~23]{bcbrar}, Ramanujan expressed his interest in \emph{Dirichlet's Divisor Problem}, although it is doubtful that he would have been familiar with this attribution. Perhaps Ramanujan had derived \eqref{b.c.11} with the intention of applying it to the \emph{Divisor Problem}.

 Berndt, Kim, and Zaharescu first proved \eqref{b.c.11}, but with the order of summation reversed and with the  assumption that the double series on the right-hand side of \eqref{b.c.11} converges for at least  one value of $\theta$ \cite{besselII}.  A complete proof of \eqref{b.c.11} was later established by Berndt and Zaharescu along with J.~Li \cite{Paper2}.  Kim established similar theorems for weighted divisor sums and Riesz sums \cite{sunkim}.  Very briefly and roughly, a Riesz sum is an extension of a sum such as \eqref{Dofx}, but with the summands now weighted by $(x-n)^{\rho}$, where  $\rho>a$, for some real number $a$.

In this paper, we establish analogues of \eqref{b.c.11} for divisor sums associated with Dedekind zeta functions attached to algebraic number fields.  As in \cite{sunkim}, our analogues are in the more general setting of Riesz sums.   Our motivation arises from unsolved problems for these divisor sums that are analogous to the \emph{Dirichlet Divisor Problem}.

   Let $K$ be an algebraic number field. For $\sigma=\Re(s)>1$, the Dedekind zeta function associated with $K$ is defined by
\begin{equation*}
\zeta_K(s):=\sum_{I\subseteq \mathcal{O}_K}\frac{1}{\big(N_{K/\mathbb{Q}}(I)\big)^s},
\end{equation*}
where $\mathcal{O}_K$ is the ring of integers of $K$, the sum is over all non-zero integral ideals $I$ of $\mathcal{O}_K,$  and $N_{K/\mathbb{Q}}(I)$ denotes the norm of $I.$
Furthermore, define $f_K(n)$ by
\begin{equation}\label{fK}
\zeta_K(s)=:\sum_{n=1}^{\infty}\frac{f_K(n)}{n^s}.
\end{equation}
 Set
\begin{equation*}
{\sum_{n\leq x}}^{\prime}f_K(n)=\gamma_{-1}(K)x+E_K(x),
\end{equation*}
where $\gamma_{-1}(K)$ is a constant defined by \eqref{res} in the following section, and where $E_K(x)$ is the ``error term.''
  E.~Landau's began this study in 1912 when he proved that  \cite{landau}
  \begin{equation*} 
{\sum_{n\leq x}}^{\prime}f_K(n)=\gamma_{-1}(K)x+O(x^{(k-1)/(k+1)}),
\end{equation*}
where $k$ is the degree of $K$, and the implied constant in the big-O term depends upon $K$.  Currently, the best results are due to W.~G.~Nowak \cite{nowak} and B.~Paul and A.~Sankaranarayanan \cite{ps}.  The former author proved that, for $k\geq3$,  
\begin{equation*}
E_K(x) =O\left(x^{1-\frac{2}{k}+\frac{3}{2k^2}}(\log x)^{\frac{2}{k}}\right), 
\end{equation*}
 while the latter two authors proved that, for $k\geq 10$,
 \begin{equation*} 
 E_K(x)=O\left(x^{1-\frac{3}{k+6}+\epsilon}\right),
 \end{equation*}
 for every $\epsilon>0$.  
 Observe that 
 $$\dfrac{k-1}{k+1}=1-\dfrac{2}{k}+\df{2}{k(k+1)},$$
 and, for $k\geq10$, 
 $$\df{3}{k+6}>\df{2}{k}-\df{3}{2k^2}=\df{4k-3}{2k^2}.$$

Our two primary theorems below, Theorem \ref{theorem3} and Theorem \ref{2ndmain}, provide analogues of \eqref{b.c.11} involving, respectively, $f_K(n)$ and $d_{\chi_D}(n)$, where $d_{\chi_D}(n):={\sum}_{k|n}\chi_D(k)$, and $D$ is a fundamental discriminant.  In Theorems \ref{theorem6.3} and \ref{theorem6.4}, we obtain upper bounds for the ``error'' terms associated with these sums.  Perhaps our identities with the additional parameter $\theta$ may be of use in attacking these long-standing unsolved problems originating with Landau.

\section{Preliminary Results}\label{prelim}
Define, for a Dirichlet character $\chi$,
\begin{align*}
F_K(x):=\sum_{n\leq x}f_K(n), \quad D_{K}(n):=\sum_{d|n}f_K(d), \quad  D_{K, \chi}(n):=\sum_{d \mid n}f_K(d)\chi(n/d).
\end{align*}
If  $\sigma>1$,
\begin{align*}
\zeta_K(s)\zeta(s)=\sum_{n=1}^{\infty}\frac{f_K(n)}{n^s}\sum_{n=1}^{\infty}\frac{1}{n^s}
=\sum_{n=1}^{\infty}\frac{D_{K}(n)}{n^s},
\end{align*}
and
\begin{align*}
\zeta_K(s)L(s,\chi)=\sum_{n=1}^{\infty}\frac{f_K(n)}{n^s}\sum_{n=1}^{\infty}\frac{\chi(n)}{n^s}
=\sum_{n=1}^{\infty}\frac{D_{K,\chi}(n)}{n^s},
\end{align*}
where $\zeta(s)$ denotes the Riemann zeta function, and $L(s, \chi)$ denotes the Dirichlet $L$-function associated with the character $\chi.$

Also,
\begin{align*}
\sum_{n\leq x}F_K\Big(\frac{x}{n}\Big)\chi(n)&=\sum_{n\leq x}\sum_{m\leq x/n}f_K(m)\chi(n)=\sum_{mn\leq x}f_K(m)\chi(n)\\
&=\sum_{n\leq x}\sum_{d \mid n}f_K(d)\chi(n/d)=\sum_{n \leq x}D_{K, \chi}(n).
\end{align*}

Next, let $\Delta_K$ denote the discriminant of $K$, and let $r_1$ and $r_2$ denote the number of real embeddings and the number of conjugate pairs of complex embeddings of $K,$ respectively.
The Dedekind zeta function $\zeta_K(s)$ has a simple pole at $s=1$ with residue \cite[p.~467]{neukirch}
\begin{equation}\label{res}
\gamma_{-1}(K):=\frac{2^{r_1+r_2}\pi^{r_2}\Reg_Kh_K}{w_K\sqrt{|\Delta_K|}},
\end{equation}
where $h_K$ is the class number of $K$, $\Reg_K$ is the regulator of $K$, and $w_K$ denotes the number of roots of unity in $K$. The Laurent expansion of $\zeta_K(s)$ around  $s=1$ is given by
\begin{equation}\label{dzetaexp}
\zeta_K(s)=\frac{\gamma_{-1}(K)}{s-1}+\sum_{n=0}^{\infty}\gamma_n(K)(s-1)^n.
\end{equation}

The function
\begin{align}\label{Lam}
\Lambda_K(s):=|\Delta_K|^{s/2}\left(\pi^{-s/2}\Gamma\big(\tf{1}{2}s\big)\right)^{r_1}\left(2(2\pi)^{-s}\Gamma(s)\right)^{r_2}\zeta_K(s)
\end{align}
satisfies the functional equation \cite[p.~466]{neukirch}
\begin{align}\label{Lamfe}
\Lambda_K(s)=\Lambda_K(1-s).
\end{align}
We also recall that the functional equation of the Riemann zeta function $\zeta(s)$ is given by \cite[p.~59]{davenport}
\begin{equation}\label{zetafe}
\pi^{-s/2}\Gamma\big(\tf12 s\big)\zeta(s)=\pi^{-(1-s)/2}\Gamma\big(\tf12 (1-s)\big)\zeta(1-s).
\end{equation}
If $\chi$ is an even, non-principal, primitive character of modulus $q$, then the Dirichlet $L$-function
$L(s,\chi)$ satisfies the functional equation
\cite[p.~69]{davenport}
\begin{equation}\label{evencha}
\Big(\frac{\pi}{q}\Big)^{-s/2}\Gamma\big(\tf12 s\big)L(s,\chi)=\df{G(\chi)}{\sqrt{q}}\Big(\frac{\pi}{q}\Big)^{-(1-s)/2}\Gamma(\tf12 (1-s))L(1-s,\overline{\chi}),
\end{equation}
where $G(\chi)$ denotes the Gauss sum
\begin{equation}\label{gauss}
G(\chi):=\sum_{h=1}^{q-1}\chi(h)e^{2\pi{i}h/q}.
\end{equation}

\begin{definition}\label{definition1}
The Meijer $G$-function is defined by \cite[p.~374]{erd1}
\begin{equation}\label{G}
G_{p,q}^{m,n}\left(x\Bigg| \begin{matrix}a_1,a_2,\dots,a_p\\b_1,b_2,\dots, b_q\end{matrix}\right)
:=\df{1}{2\pi i}\int_L\df{\underset{j=1}{{\overset{m}{\prod}}}\Gamma(b_j-s)\underset{j=1}{{\overset{n}{\prod}}}\Gamma(1-a_j+s)}
{\underset{j=m+1}{{\overset{q}{\prod}}}\Gamma(1-b_j+s)\underset{j=n+1}{{\overset{p}{\prod}}}\Gamma(a_j-s)}x^s \,ds,
\end{equation}
where $L$ is a path from $-i\infty$ to $+i\infty$ separating the poles of $\Gamma(b_1-s)\cdots\Gamma(b_m-s)$ from those of
 $\Gamma(1-a_1+s)\cdots\Gamma(1-a_n+s)$.
 When the arguments $a_1,a_2,\dots, a_p; b_1,b_2,\dots,b_q$ are omitted, we write $
 G_{p,q}^{m,n}$ or $G_{p,q}^{m,n}(x)$.
 
\end{definition}

\section{The Case: $r_2=0$}
We consider the case $r_2=0.$ Let $\chi$ be a non-principal, even, primitive character of modulus $q.$
By \eqref{Lam}, \eqref{Lamfe}, and \eqref{evencha}, the functional equation of $\zeta_K(2s)L(2s,\chi)$ is given by
\begin{align}\label{feq}
&\pi^{-s(r_1+1)}q^{s}|\Delta_K|^{s}\Gamma(s)^{r_1+1}\zeta_K(2s)L(2s,\chi) \notag \\
&\qquad =\frac{G(\chi)}{\sqrt{q}}\pi^{-\big(\frac{1}{2}-s\big)(r_1+1)}q^{\frac12-s}|\Delta_K|^{\frac{1}{2}-s}\Gamma\big(\tf {1}{2}-s\big)^{r_1+1}\zeta_K(1-2s)L(1-2s,\overline{\chi}).
\end{align}

Below and in the sequel it is tacitly assumed that if $\rho=0$, only one-half of the last term in a finite sum is counted.

\begin{theorem}\label{rsub1} Let $q$ be a positive integer, and let $\chi$ be a non-principal even primitive character modulo $q.$
Then, for $x>0$ and $\rho>\tf12 r_1 -1,$
\begin{align}\label{t1}
&\frac{1}{\Gamma(\rho+1)}
\sum_{n\leq x}D_{K, \chi}(n)(x^2-n^2)^{\rho} \notag\\
&=\frac{G(\chi)|\Delta_K|^{1/2}x^{2\rho}}{\pi^{(r_1+1)/2}}\sum_{n=1}^{\infty}\frac{D_{K, \overline{\chi}}(n)}{n}
G_{2(r_1+1),0}^{0,r_1+1}\Big(\frac{q^2|\Delta_K|^2}{\pi^{2(r_1+1)}(nx)^2}
\Big| \begin{matrix}\tf12,\tf12,\dots, \tf12,\rho+1,0,\dots, 0\\-\end{matrix}\Big)\notag\\
&\qquad -\frac{G(\chi)\sqrt{\pi}\gamma_{-1}(K)x^{1+2\rho}}{2q\Gamma(\rho+\tf32)}
\sum_{n=1}^{q-1}\overline{\chi}(n)\log{(2\sin(\pi n/q))},
\end{align}
where $G_{2(r_1+1),0}^{0,r_1+1}$ is defined in Definition \ref{definition1}, and where there are $r_1$\, $\{0\}$'s and $r_1+1$\, $\{\tf12\}$'s in $G_{2(r_1+1),0}^{0,r_1+1}$.
\end{theorem}

\begin{proof}
Using \eqref{feq}, we apply Theorems 2 and 4 in \cite[pp.~351, 356]{identitiesI} with, in the notation of \cite{identitiesI}, 
\begin{equation*}
m=r_1+1, ~ r=\frac12,
~ \lambda_n=\mu_n=\dfrac{\pi^{r_1+1}n^2}{q|\Delta_K|},
~a_n=D_{K, \chi}(n), ~ b_n=\dfrac{G(\chi)}{\sqrt{q}}D_{K, \overline{\chi}}(n).
\end{equation*}
Then,  for $x>0$ and $\rho>\tf12 r_1-1$,
\begin{gather}
\frac{1}{\Gamma(\rho+1)}
\sum_{\lambda_n\leq x}D_{K, \chi}(n)(x-\lambda_n)^{\rho} \notag\\
=\frac{G(\chi)}{2^{r_1\rho}\sqrt{q}}
\sum_{n=1}^{\infty}D_{K, \overline{\chi}}(n)\Big(\frac{x}{\mu_n}\Big)^{\frac14+\frac{\rho}{2}}
K_{\frac12+\rho}\big(2^{r_1+1}\sqrt{\mu_{n}x};-\tf12;r_1+1\big)+Q_{\rho}(x).\label{Eq1}
\end{gather}
Here,
\begin{equation*}
Q_{\rho}(x)=\df{1}{2\pi{i}}\int_{C_{\rho}}q^{s}|\Delta_K|^{s}\df{\Gamma(s)\zeta_K(2s)L(2s,\chi)x^{s+\rho}}{\pi^{s(r_1+1)}\Gamma(s+\rho+1)}ds,
\end{equation*}
where $C_{\rho}$ is a positively oriented, closed curve encircling the poles of the integrand,
\begin{align}\label{K}
K_{\tf12+\rho}(x;-\tf12;r_1+1)&=\int_0^{\i}u_{r_1}^{\rho}J_{-\tf12}(u_{r_1})du_{r_1}\int_0^{\i}u_{r_1-1}^{\rho}J_{-\tf12}(u_{r_1-1})du_{r_1-1} \notag\\
&\qquad \qquad \qquad  \ldots\ldots \int_0^{\i}u_{1}^{\rho}J_{-\tf12}(u_{1})J_{\tf12+\rho}(x/u_1u_2\cdots u_{r_1})du_{1}
\end{align}
 \cite[p.~348, Definition 4]{identitiesI}, and $J_{\nu}(z)$ denotes the ordinary Bessel function of order $\nu$.

We first calculate $Q_{\rho}(x).$  The left-hand side of \eqref{evencha} is an entire function, and so $L(0, \chi)=0.$ Therefore, the only pole of the integrand is at $s=\tf12,$ arising from the simple pole of $\zeta_K(2s).$
Thus, using \eqref{res} and $\Gamma(\tf12)=\sqrt{\pi}$, we have
\begin{align*}
Q_{\rho}(x)&=\frac{\sqrt{q|\Delta_K|}\gamma_{-1}(K)x^{\frac12+\rho}}{2\pi^{r_1/2}\Gamma(\rho+\tf32)}L(1, \chi).
\end{align*}
Using an evaluation for $L(1, \chi)$ in \cite[p.~182, Equation (3.5)]{CD}, we obtain
\begin{align}\label{Q}
Q_{\rho}(x)&=-\frac{G(\chi)\sqrt{|\Delta_K|}\gamma_{-1}(K)x^{\frac12+\rho}}{2\pi^{r_1/2}\sqrt{q}\Gamma(\rho+\tf32)}
\sum_{n=1}^{q-1}\overline{\chi}(n)\log{|1-\xi_q^n|} \notag\\
&=-\frac{G(\chi)\sqrt{|\Delta_K|}\gamma_{-1}(K)x^{\frac12+\rho}}{2\pi^{r_1/2}\sqrt{q}\Gamma(\rho+\tf32)}
\sum_{n=1}^{q-1}\overline{\chi}(n)\log(2\sin(\pi n/q)),
\end{align}
where we used the identity
\begin{align*}
\log{|1-\xi_q^n|}=\log{|\xi_q^{-n/2}-\xi_q^{n/2}|} =\log{(2\sin(\pi n/q))}.
\end{align*}
Lastly, we employ the identity
\begin{align}\label{mainmain}
K_{\tf12+\rho}(x;-\tf12;r_1+1)&=2^{-\rho+(r_1+1)/2}x^{\rho-1/2}
G_{2(r_1+1),0}^{0,r_1+1}\Big(\frac{4^{r_1+1}}{x^2}\Big| \begin{matrix}\tf12,\tf12,\dots, \tf12,\rho+1,0,\dots, 0\\-\end{matrix}\Big),
\end{align}
where there are $r_1\,\{0\} $'s and $(r_1+1)\, \{\tf12\}$'s,
which we prove in Theorem \ref{rsub2}.

Now, we replace $x$ by $\tf{\pi^{r_1+1}x^2}{q|\Delta_K|}$ in \eqref{Eq1}. Then, by \eqref{Q} and \eqref{mainmain}, we have
\begin{align*}
&\frac{1}{\Gamma(\rho+1)}
\sum_{n\leq x}D_{K, \chi}(n)(x^2-n^2)^{\rho} \notag\\
&=\frac{G(\chi)(q|\Delta_K|)^{\rho}}{(2^{r_1}\pi^{r_1+1})^\rho\sqrt{q}}
\sum_{n=1}^{\infty}D_{K, \overline{\chi}}(n)\Big(\frac{x}{n}\Big)^{\frac12+\rho}
K_{\frac12+\rho}\Big(\frac{(2\pi)^{r_1+1}nx}{q|\Delta_K|};-\tf12;r_1+1\Big) \notag\\
&\qquad -\frac{G(\chi)\sqrt{\pi}\gamma_{-1}(K)x^{1+2\rho}}{2q\Gamma(\rho+\tf32)}
\sum_{n=1}^{q-1}\overline{\chi}(n)\log{(2\sin(\pi n/q))}\notag\\
&=\frac{G(\chi)|\Delta_K|^{1/2}x^{2\rho}}{\pi^{(r_1+1)/2}}\sum_{n=1}^{\infty}\frac{D_{K, \overline{\chi}}(n)}{n}
G_{2(r_1+1),0}^{0,r_1+1}\Big(\frac{q^2|\Delta_K|^2}{\pi^{2(r_1+1)}(nx)^2} 
\Big| \begin{matrix}\tf12,\tf12,\dots, \tf12,\rho+1,0,\dots, 0\\-\end{matrix}\Big)\notag\\
&\qquad -\frac{G(\chi)\sqrt{\pi}\gamma_{-1}(K)x^{1+2\rho}}{2q\Gamma(\rho+\tf32)}
\sum_{n=1}^{q-1}\overline{\chi}(n)\log{(2\sin(\pi n/q))}.
\end{align*}
This completes the proof. 
\end{proof}

The following theorem is an analogue of Vorono\"{\dotlessi}'s formula \eqref{b.vor}.  Note that if $\rho=0$ and $r_1=1$, Theorem \ref{theorem2} reduces to \eqref{b.vor}.  See also \eqref{specialcase} below.

\begin{theorem}\label{theorem2} For $x>0$ and $\rho>\tf12 r_1 -1,$
\begin{align}\label{t2}
&\frac{1}{\Gamma(\rho+1)}
\sum_{n\leq x}D_{K}(n)(x^2-n^2)^{\rho} \notag\\
&=\frac{|\Delta_K|^{1/2}x^{2\rho}}{\pi^{(r_1+1)/2}}\sum_{n=1}^{\infty}\frac{D_{K}(n)}{n}
G_{2(r_1+1),0}^{0,r_1+1}\Big(\frac{|\Delta_K|^2}{\pi^{2(r_1+1)}(nx)^2}
\Big| \begin{matrix}\tf12,\tf12,\dots, \tf12,\rho+1,0,\dots, 0\\-\end{matrix}\Big)\notag\\
&\qquad +\frac{\sqrt{\pi}x^{1+2\rho}}{2\Gamma(\rho+\tf32)}\left\{
\gamma_0(K)+\gamma_{-1}(K)\gamma+\frac{\gamma_{-1}(K)\Gamma'(\tf12)}{2\sqrt{\pi}}\right.\notag \\
&\qquad\qquad\qquad\qquad \qquad    \left.-\frac12\gamma_{-1}(K)\psi(\rho+\tf32)+\gamma_{-1}(K)\log x\right\},
\end{align}
where $\psi(s):=\dfrac{\Gamma'(s)}{\Gamma(s)},$ $\gamma$ denotes Euler's constant, and $\gamma_{-1}(K)$ and $\gamma_0(K)$ are defined in
\eqref{res} and \eqref{dzetaexp}, respectively. Also, if $r_1=1,$ then we have the following additional term on the right-hand side. 
\begin{equation*}
\frac{x^{2\rho}}{4\Gamma(\rho+1)}.
\end{equation*}
\end{theorem}

\begin{proof}
From \eqref{Lam}, \eqref{Lamfe}, and \eqref{zetafe}, we see that $\zeta_K(2s)\zeta(2s)$ satisfies the functional equation
\begin{align}\label{feq2}
&\pi^{-s(r_1+1)}|\Delta_K|^{s}\Gamma(s)^{r_1+1}\zeta_K(2s)\zeta(2s) \notag \\
&\qquad =\pi^{-\big(\frac{1}{2}-s\big)(r_1+1)}|\Delta_K|^{\frac{1}{2}-s}\Gamma\big(\tf {1}{2}-s\big)^{r_1+1}\zeta_K(1-2s)\zeta(1-2s).
\end{align}
Similarly, we apply Theorems 2 and 4 in \cite{identitiesI} with $$m=r_1+1, ~ r=\frac12,
~ \lambda_n=\mu_n=\dfrac{\pi^{r_1+1}n^2}{|\Delta_K|},
~a_n=b_n=D_{K}(n).$$
Then, we have,
for $x>0$ and $\rho> \tf12 r_1 -1$,
\begin{align}\label{Eq3}
\frac{1}{\Gamma(\rho+1)}
&\sum_{\lambda_n\leq x}D_{K}(n)(x-\lambda_n)^{\rho} \notag\\
&=2^{-r_1\rho}
\sum_{n=1}^{\infty}D_{K}(n)\Big(\frac{x}{\mu_n}\Big)^{\frac14+\frac{\rho}{2}}
K_{\frac12+\rho}\big(2^{r_1+1}\sqrt{\mu_{n}x};-\tf12;r_1+1\big)+Q_{\rho}(x),
\end{align}
where
\begin{equation*}
Q_{\rho}(x)=\df{1}{2\pi{i}}\int_{C_{\rho}} \df{|\Delta_K|^{s}\Gamma(s)\zeta_K(2s)\zeta(2s)x^{s+\rho}}{\pi^{s(r_1+1)}\Gamma(s+\rho+1)}\, ds,
\end{equation*}
and 
where $C_{\rho}$ is a positively oriented, closed curve encircling the poles of the integrand.
Note that each of $\zeta_K(2s)$ and $\zeta(2s)$ has a simple pole at $s=\tf12.$ 
If $r_1=1,$ then the integrand also has a simple pole at $s=0$ arising from $\Gamma(s).$ 
From the following Laurent expansions
\begin{align}
\zeta(2s)&=\frac{\tf12}{s-\tf12}+\gamma+\sum_{n=1}^{\infty}\frac{(-1)^n\gamma_n}{n!}2^n(s-\tf12)^n,\label{zle}\\
\zeta_K(2s)&=\frac{\tf12\gamma_{-1}(K)}{s-\tf12}+\sum_{n=0}^{\infty}\gamma_n(K)2^n(s-\tf12)^n,\\
\Gamma(s)&=\sqrt{\pi}+\Gamma'(\tf12)(s-\tf12)+\cdots,\label{Gle}\\
\frac{1}{\Gamma(s+\rho+1)}&=\frac{1}{\Gamma(\rho+\tf32)}-\frac{\psi(\rho+\tf32)}{\Gamma(\rho+\tf32)}(s-\tf12)+\cdots,\label{1Gle}\\
\Big(\frac{|\Delta_K|x}{\pi^{r_1+1}}\Big)^s&=\Big(\frac{|\Delta_K|x}{\pi^{r_1+1}}\Big)^{1/2}+
\Big(\frac{|\Delta_K|x}{\pi^{r_1+1}}\Big)^{1/2}\log\Big(\frac{|\Delta_K|x}{\pi^{r_1+1}}\Big)(s-\tf12)+\cdots,\notag
\end{align}
we derive
\begin{align}\label{Q2}
Q_{\rho}(x)&=\frac{\sqrt{|\Delta_K|}x^{\rho+\tf12}}{2\Gamma(\rho+\tf32)\pi^{r_1/2}}\left\{
\gamma_0(K)+\gamma_{-1}(K)\gamma+\frac{\gamma_{-1}(K)\Gamma'(\tf12)}{2\sqrt{\pi}}\right.\notag \\
&\qquad\qquad  \left.-\frac12\gamma_{-1}(K)\psi(\rho+\tf32)+\frac12\gamma_{-1}(K)\log\Big(\frac{|\Delta_K|x}{\pi^{r_1+1}}\Big)\right\},
\end{align}
where if $r_1=1,$  $Q_{\rho}(x)$ has an additional term
$$\frac{x^{\rho}}{4\Gamma(\rho+1)}.$$

Now, replacing $x$ by $\tf{\pi^{r_1+1}x^2}{|\Delta_K|}$ in \eqref{Eq3} and using \eqref{Q2} and \eqref{mainmain}, we obtain for $r_1\geq 2,$
\begin{align}
&\frac{1}{\Gamma(\rho+1)}\sum_{n\leq x}D_K(n)\big(x^2-n^2\big)^{\rho}\notag\\
&=\frac{|\Delta_K|^{\rho}}{2^{r_1\rho}\pi^{(r_1+1)\rho}}\sum_{n=1}^{\infty}D_K(n)\Big(\frac{x}{n}\Big)^{\rho+1/2}
K_{\frac12+\rho}\Big(\frac{(2\pi)^{r_1+1}nx}{|\Delta_K|};-\tf12;r_1+1\Big)\notag\\
&\qquad \qquad +\frac{\sqrt{\pi}x^{1+2\rho}}{2\Gamma(\rho+\tf32)}\left\{
\gamma_0(K)+\gamma_{-1}(K)\gamma+\frac{\gamma_{-1}(K)\Gamma'(\tf12)}{2\sqrt{\pi}}\right.
\notag \\
&\qquad\qquad\qquad\qquad \qquad    \left.-\frac12\gamma_{-1}(K)\psi(\rho+\tf32)+\gamma_{-1}(K)\log x\right\}\label{star}\\
&=\frac{|\Delta_K|^{1/2}x^{2\rho}}{\pi^{(r_1+1)/2}}\sum_{n=1}^{\infty}\frac{D_{K}(n)}{n}
G_{2(r_1+1),0}^{0,r_1+1}\Big(\frac{|\Delta_K|^2}{\pi^{2(r_1+1)}(nx)^2}
\Big| \begin{matrix}\tf12,\tf12,\dots, \tf12,\rho+1,0,\dots, 0\\-\end{matrix}\Big)\notag\\
&\qquad \qquad +\frac{\sqrt{\pi}x^{1+2\rho}}{2\Gamma(\rho+\tf32)}\left\{
\gamma_0(K)+\gamma_{-1}(K)\gamma+\frac{\gamma_{-1}(K)\Gamma'(\tf12)}{2\sqrt{\pi}}\right.
\notag \\
&\qquad\qquad\qquad\qquad \qquad    \left.-\frac12\gamma_{-1}(K)\psi(\rho+\tf32)+\gamma_{-1}(K)\log x\right\}.\notag
\end{align}
For $r_1=1,$  we have the following additional term on the right-hand side of the equation above. 
$$\frac{x^{2\rho}}{4\Gamma(\rho+1)}.$$
This completes the proof.
\end{proof}
For convenience, we employ the abbreviation from Definition \ref{definition1}, 
$$G_{2(r_1+1),0}^{0,r_1+1}(x):
=G_{2(r_1+1),0}^{0,r_1+1}\Big(x\Big| \begin{matrix}\tf12,\tf12,\dots, \tf12,\rho+1,0,\dots, 0\\-\end{matrix}\Big).$$

\begin{theorem}\label{theorem3} For $x>0,$ $0<\theta<1$, and $\rho>\tf12 r_1 -1,$
\begin{align}\label{mainmain33}
&\sum_{n\leq x}(x^2-n^2)^{\rho}\sum_{r\mid n}f_K\Big(\frac{n}{r}\Big)\cos(2\pi r\theta) \notag\\
&=\frac{\Gamma(\rho+1)|\Delta_K|^{1/2}x^{2\rho}}{2\pi^{(r_1+1)/2}}\sum_{m=1}^{\infty}\frac{f_K(m)}{m}\notag\\
&\qquad \times
\sum_{n=0}^{\infty}\left\{\frac{G_{2(r_1+1),0}^{0,r_1+1}\Big(\frac{|\Delta_K|^2}{\pi^{2(r_1+1)}(m(n+\theta)x)^2}\Big)}{n+\theta}
+\frac{G_{2(r_1+1),0}^{0,r_1+1}\Big(\frac{|\Delta_K|^2}{\pi^{2(r_1+1)}(m(n+1-\theta)x)^2}\Big)}{n+1-\theta}\right\}\notag\\
&\qquad-\frac{\Gamma(\rho+1)\sqrt{\pi}\gamma_{-1}(K)x^{1+2\rho}}{2\Gamma(\rho+\tf32)}\log{(2\sin(\pi\theta))},
\end{align}
where $f_K(n)$ is defined in \eqref{fK}.
\end{theorem}

\begin{proof} It suffices to prove \eqref{mainmain33} for $\theta=h/q,$ where $q$ is prime and $0<h<q.$
Using  the identity \cite [Lemma 2.5]{bessel2bbskaz}
\begin{align*}
\cos\Big(\frac{2\pi ha}{q} \Big)=\frac{1}{\phi(q)}\sum_{\substack{\chi \bmod{q}\\ \chi~\textup{even}}}\chi(a)\chi(h)G(\overline{\chi}),
\end{align*}
where $\phi(q)$ is the Euler's $\phi$-function and $G(\chi)$ denotes the Gauss sum in \eqref{gauss},
we deduce that
\begin{align}\label{fKcos}
&\sum_{n\leq x}(x^2-n^2)^{\rho}\sum_{r\mid n}f_K\Big(\frac{n}{r}\Big)\cos\Big(\frac{2\pi rh}{q}\Big) \notag\\
&=\sum_{n\leq x}(x^2-n^2)^{\rho}\left\{\sum_{\substack{r\mid n\\q\mid r}}f_K\Big(\frac{n}{r}\Big)+
\sum_{\substack{r\mid n\\q\nmid r}}f_K\Big(\frac{n}{r}\Big)\cos\Big(\frac{2\pi rh}{q}\Big)\right\}\notag\\
&=\sum_{n\leq x}(x^2-n^2)^{\rho}\left\{\sum_{qm\mid n}f_K\Big(\frac{n}{qm}\Big)+
\frac{1}{\phi(q)}\sum_{\substack{r\mid n\\q\nmid r}}f_K\Big(\frac{n}{r}\Big)
\sum_{\substack{\chi \\ \text{even}}}\chi(h)\chi(r)G(\overline{\chi})\right\}\notag\\
&=\sum_{n\leq x/q}(x^2-q^2n^2)^{\rho}\sum_{d\mid n}f_K(d)
+\frac{1}{\phi(q)}\sum_{\substack{\chi \\ \text{even}}}\chi(h)G(\overline{\chi})
\sum_{n\leq x}(x^2-n^2)^{\rho}\sum_{r\mid n}f_K\Big(\frac{n}{r}\Big)\chi(r) \notag\\
&=q^{2\rho}\sum_{n\leq x/q}\left(\Big(\frac{x}{q}\Big)^2-n^2\right)^{\rho}D_K(n)
+\frac{1}{\phi(q)}\sum_{\substack{\chi \\ \text{even}}}\chi(h)G(\overline{\chi})
\sum_{n\leq x}(x^2-n^2)^{\rho}D_{K, \chi}(n) \notag\\
&=q^{2\rho}\sum_{n\leq x/q}\left(\Big(\frac{x}{q}\Big)^2-n^2\right)^{\rho}D_K(n)
-\frac{1}{\phi(q)}\sum_{n\leq x}(x^2-n^2)^{\rho}D_{K, \chi_0}(n)\notag\\
&\qquad \qquad+\frac{1}{\phi(q)}\sum_{\substack{\chi\neq\chi_0 \\ \text{even}}}\chi(h)G(\overline{\chi})
\sum_{n\leq x}(x^2-n^2)^{\rho}D_{K, \chi}(n),
\end{align}
where $\chi_0$ denotes the principal character modulo $q$. 

Note that
\begin{align}\label{DKchi0}
&\sum_{n\leq x}(x^2-n^2)^{\rho}D_{K, \chi_0}(n)\notag\\
&\qquad =\sum_{n\leq x}(x^2-n^2)^{\rho}\sum_{r\mid n}f_K\Big(\frac{n}{r}\Big)\chi(r) \chi_0(r) \notag\\
&\qquad =\sum_{n\leq x}(x^2-n^2)^{\rho}\sum_{\substack{r\mid n\\q\nmid r}}f_K\Big(\frac{n}{r}\Big) \notag\\
&\qquad=\sum_{n\leq x}(x^2-n^2)^{\rho}\left\{\sum_{r\mid n}f_K\Big(\frac{n}{r}\Big)-\sum_{q\mid r\mid n}f_K\Big(\frac{n}{r}\Big)\right\}\notag\\
&\qquad=\sum_{n\leq x}(x^2-n^2)^{\rho}\sum_{r\mid n}f_K\Big(\frac{n}{r}\Big)-
\sum_{n\leq x/q}(x^2-q^2n^2)^{\rho}\sum_{d\mid n}f_K\Big(\frac{n}{d}\Big)  \notag\\
&\qquad=\sum_{n\leq x}(x^2-n^2)^{\rho}D_K(n)-q^{2\rho}\sum_{n\leq x/q}\left(\Big(\frac{x}{q}\Big)^2-n^2\right)^{\rho}D_K(n).
\end{align}
Thus, using \eqref{t1}, \eqref{t2}, and \eqref{DKchi0} in \eqref{fKcos}, we have
{\allowdisplaybreaks\begin{align}
&\sum_{n\leq x}(x^2-n^2)^{\rho}\sum_{r\mid n}f_K\Big(\frac{n}{r}\Big)\cos\Big(\frac{2\pi rh}{q}\Big) \notag\\
=&\frac{q^{2\rho+1}}{\phi(q)}\sum_{n\leq x/q}\left(\Big(\frac{x}{q}\Big)^2-n^2\right)^{\rho}D_K(n)
-\frac{1}{\phi(q)}\sum_{n\leq x}(x^2-n^2)^{\rho}D_K(n)\label{fKcos2}\\
&\qquad+\frac{1}{\phi(q)}\sum_{\substack{\chi\neq\chi_0 \\ \text{even}}}\chi(h)G(\overline{\chi})
\sum_{n\leq x}(x^2-n^2)^{\rho}D_{K, \chi}(n)\notag\\
=&\frac{q^{2\rho+1}\Gamma(\rho+1)}{\phi(q)}\left\{\frac{|\Delta_K|^{1/2}x^{2\rho}}{q^{2\rho}\pi^{(r_1+1)/2}}\sum_{n=1}^{\infty}\frac{D_{K}(n)}{n}
G_{2(r_1+1),0}^{0,r_1+1}\Big(\frac{|\Delta_K|^2q^2}{\pi^{2(r_1+1)}(nx)^2}\Big)\right.\notag\\
&\qquad \qquad +\left.\frac{\sqrt{\pi}x^{1+2\rho}\gamma_{-1}(K)}{2q^{2\rho+1}\Gamma(\rho+\tf32)}
\Big(\gamma+\frac{\gamma_0(K)}{\gamma_{-1}(K)}+
\frac{\Gamma'(\tf12)}{2\sqrt{\pi}}-\frac12\psi(\rho+\tf32)+\log \big(\tf{x}{q}\big)
\Big)\right\}\notag\\
&-\frac{\Gamma(\rho+1)}{\phi(q)}\left\{\frac{|\Delta_K|^{1/2}x^{2\rho}}{\pi^{(r_1+1)/2}}\sum_{n=1}^{\infty}\frac{D_{K}(n)}{n}
G_{2(r_1+1),0}^{0,r_1+1}\Big(\frac{|\Delta_K|^2}{\pi^{2(r_1+1)}(nx)^2}\Big)\right.\notag\\
&\qquad \qquad +\left.\frac{\sqrt{\pi}x^{1+2\rho}\gamma_{-1}(K)}{2\Gamma(\rho+\tf32)}
\Big(\gamma+\frac{\gamma_0(K)}{\gamma_{-1}(K)}+
\frac{\Gamma'(\tf12)}{2\sqrt{\pi}}-\frac12\psi(\rho+\tf32)+\log x\Big)\right\}\notag\\
&+\frac{\Gamma(\rho+1)}{\phi(q)}\sum_{\substack{\chi\neq\chi_0 \\ \text{even}}}\chi(h)G(\overline{\chi})\notag\\
&\times
\left\{\frac{G(\chi)|\Delta_K|^{1/2}x^{2\rho}}{\pi^{(r_1+1)/2}}\sum_{n=1}^{\infty}\frac{D_{K, \overline{\chi}}(n)}{n}
G_{2(r_1+1),0}^{0,r_1+1}\Big(\frac{q^2|\Delta_K|^2}{\pi^{2(r_1+1)}(nx)^2}\Big)\right.\notag\\
&\qquad \left.-\frac{G(\chi)\sqrt{\pi}\gamma_{-1}(K)x^{1+2\rho}}{2q\Gamma(\rho+\tf32)}
\sum_{n=1}^{q-1}\overline{\chi}(n)\log{(2\sin(\pi n/q))}\right\}\notag\\
=&\frac{\Gamma(\rho+1)|\Delta_K|^{1/2}x^{2\rho}q}{\phi(q)\pi^{(r_1+1)/2}}\sum_{n=1}^{\infty}\frac{D_{K}(n)}{n}
G_{2(r_1+1),0}^{0,r_1+1}\Big(\frac{|\Delta_K|^2q^2}{\pi^{2(r_1+1)}(nx)^2}\Big)\notag\\
&-\frac{\Gamma(\rho+1)|\Delta_K|^{1/2}x^{2\rho}}{\phi(q)\pi^{(r_1+1)/2}}\sum_{n=1}^{\infty}\frac{D_{K}(n)}{n}
G_{2(r_1+1),0}^{0,r_1+1}\Big(\frac{|\Delta_K|^2}{\pi^{2(r_1+1)}(nx)^2}\Big)\notag\\
&-\frac{\Gamma(\rho+1)\sqrt{\pi}\gamma_{-1}(K)x^{1+2\rho}}{2\phi(q)\Gamma(\rho+\tf32)}\log q\notag \\
&+\frac{\Gamma(\rho+1)q}{\phi(q)}\sum_{\substack{\chi \\ \text{even}}}\chi(h)\notag\\
&\times
\left\{\frac{|\Delta_K|^{1/2}x^{2\rho}}{\pi^{(r_1+1)/2}}\sum_{n=1}^{\infty}\frac{D_{K, \overline{\chi}}(n)}{n}
G_{2(r_1+1),0}^{0,r_1+1}\Big(\frac{q^2|\Delta_K|^2}{\pi^{2(r_1+1)}(nx)^2}\Big)\right.\notag\\
&\qquad \left.-\frac{\sqrt{\pi}\gamma_{-1}(K)x^{1+2\rho}}{2q\Gamma(\rho+\tf32)}
\sum_{n=1}^{q-1}\overline{\chi}(n)\log{(2\sin(\pi n/q))}\right\}\notag\\
&-\frac{\Gamma(\rho+1)q}{\phi(q)}\frac{|\Delta_K|^{1/2}x^{2\rho}}{\pi^{(r_1+1)/2}}\sum_{n=1}^{\infty}\frac{D_{K, \overline{\chi}_0}(n)}{n}
G_{2(r_1+1),0}^{0,r_1+1}\Big(\frac{q^2|\Delta_K|^2}{\pi^{2(r_1+1)}(nx)^2}\Big)\notag\\
&+\frac{\Gamma(\rho+1)}{\phi(q)}\frac{\sqrt{\pi}\gamma_{-1}(K)x^{1+2\rho}}{2\Gamma(\rho+\tf32)}
\sum_{n=1}^{q-1}\log{(2\sin(\pi n/q))}.
\end{align}}
Employing the identity,
 $$D_{K, \overline{\chi}_0}(n)=D_K(n)-D_K\Big(\frac{n}{q}\Big),$$
  the formula \cite[p. 41, Formula 1.392, no. 1]{grad},
 \begin{equation}\label{sin}
\prod_{n=1}^{q-1}\sin\Big(\frac{\pi n}{q}\Big)=\frac{q}{2^{q-1}},
\end{equation}
and \cite[equation (3.7)]{besselII}
\begin{equation}\label{chieven}
\sum_{\substack{\chi \\ \text{even}}}\chi(a)\overline{\chi}(h)=
\begin{cases} \phi(q)/2 & \text{if}, \, \,  h=\pm a \pmod{q},\\
0, & \text{otherwise},
\end{cases}
\end{equation}
we find that
{\allowdisplaybreaks\begin{align*}
&\sum_{n\leq x}(x^2-n^2)^{\rho}\sum_{r\mid n}f_K\Big(\frac{n}{r}\Big)\cos\Big(\frac{2\pi rh}{q}\Big) \notag\\
=&\frac{\Gamma(\rho+1)q|\Delta_K|^{1/2}x^{2\rho}}{\phi(q)\pi^{(r_1+1)/2}}\sum_{\substack{\chi \\ \text{even}}}\chi(h)
\sum_{n=1}^{\infty}\frac{D_{K, \overline{\chi}}(n)}{n}
G_{2(r_1+1),0}^{0,r_1+1}\Big(\frac{q^2|\Delta_K|^2}{\pi^{2(r_1+1)}(nx)^2}\Big)\notag\\
&-\frac{\Gamma(\rho+1)\sqrt{\pi}\gamma_{-1}(K)x^{1+2\rho}}{2\phi(q)\Gamma(\rho+\tf32)}\log q\\
&\qquad -\frac{\Gamma(\rho+1)\sqrt{\pi}\gamma_{-1}(K)x^{1+2\rho}}{2\phi(q)\Gamma(\rho+\tf32)}
\left\{\sum_{\substack{\chi \\ \text{even}}}\chi(h)
\sum_{n=1}^{q-1}\overline{\chi}(n)\log{(2\sin(\pi n/q))}-\log q\right\}\notag \\
=&\frac{\Gamma(\rho+1)q|\Delta_K|^{1/2}x^{2\rho}}{\phi(q)\pi^{(r_1+1)/2}}\sum_{\substack{\chi \\ \text{even}}}\chi(h)
\sum_{n=1}^{\infty}\frac{1}{n}\sum_{d\mid n}f_K\Big(\frac{n}{d}\Big)\overline{\chi}(d)
G_{2(r_1+1),0}^{0,r_1+1}\Big(\frac{q^2|\Delta_K|^2}{\pi^{2(r_1+1)}(nx)^2}\Big)\notag\\
&\qquad -\frac{\Gamma(\rho+1)\sqrt{\pi}\gamma_{-1}(K)x^{1+2\rho}}{2\phi(q)\Gamma(\rho+\tf32)}
\left\{\sum_{n=1}^{q-1}\log{(2\sin(\pi n/q))}\sum_{\substack{\chi \\ \text{even}}}\chi(h)\overline{\chi}(n)\right\}\notag \\
=&\frac{\Gamma(\rho+1)q|\Delta_K|^{1/2}x^{2\rho}}{\phi(q)\pi^{(r_1+1)/2}}\sum_{m, d=1}^{\infty}\frac{f_K(m)}{md}
\sum_{\substack{\chi \\ \text{even}}}\chi(h)\overline{\chi}(d)
G_{2(r_1+1),0}^{0,r_1+1}\Big(\frac{q^2|\Delta_K|^2}{\pi^{2(r_1+1)}(mdx)^2}\Big)\notag\\
&\qquad-\frac{\Gamma(\rho+1)\sqrt{\pi}\gamma_{-1}(K)x^{1+2\rho}}{2\Gamma(\rho+\tf32)}\log{(2\sin(\pi h/q))}\notag\\
=&\frac{\Gamma(\rho+1)|\Delta_K|^{1/2}x^{2\rho}}{2\pi^{(r_1+1)/2}}\sum_{m=1}^{\infty}\frac{f_K(m)}{m}\notag\\
&\qquad \times
\sum_{n=0}^{\infty}\left\{\frac{G_{2(r_1+1),0}^{0,r_1+1}\Big(\frac{|\Delta_K|^2}{\pi^{2(r_1+1)}(m(n+h/q)x)^2}\Big)}{n+h/q}
+\frac{G_{2(r_1+1),0}^{0,r_1+1}\Big(\frac{|\Delta_K|^2}{\pi^{2(r_1+1)}(m(n+1-h/q)x)^2}\Big)}{n+1-h/q}\right\}\notag\\
&\qquad-\frac{\Gamma(\rho+1)\sqrt{\pi}\gamma_{-1}(K)x^{1+2\rho}}{2\Gamma(\rho+\tf32)}\log{(2\sin(\pi h/q))}.
\end{align*}}
This completes the proof.
\end{proof}

\section{Dirichlet's Theorem}
We state a classical theorem of Dirichlet \cite{dirichlet}.

\begin{theorem}\label{dir} Let $\left(\tf{\cdot}{m}\right)$ denote the Kronecker symbol, where $m$ is any positive integer \cite[p.~39]{davenport}.  Let $n$ be a positive integer coprime with the discriminant $D$ of a positive-definite primitive integral binary quadratic form.  Let $R_{D}(n)$ denote the number of all representations of $n$ by a representative set of positive-definite primitive integral binary quadratic forms of discriminant $D$.  Then
\begin{equation}\label{R}
R_D(n)=w_D\sum_{k|n}\left(\df{D}{k}\right),
\end{equation}
where
\begin{equation}\label{w}
w_D=\begin{cases}2, &\text{if \, $D<-4,$}\\
4, &\text{if \, $D=-4$,}\\
6, &\text{if \, $D=-3$}.
\end{cases}
\end{equation}
\end{theorem}

The following is Theorem 2.2.15 in \cite{HC}.
\begin{theorem}\label{theorem4.2}
If $D$ is a fundamental discriminant, the Kronecker symbol $\left(\tf{D}{n}\right)$ defines a real primitive character modulo $|D|.$
Conversely, if $\chi$ is a real  primitive character modulo $m,$ then $D=\chi(-1)m$ is a fundamental discriminant
and $\chi(n)=\left(\tf{D}{n}\right)$.
\end{theorem}

Thus, if $D$ is a fundamental discriminant,  we can rewrite $R_D(n)$ as
\begin{equation*}
R_D(n)=w_D\sum_{k|n}\chi_{D}(k)=w_Dd_{\chi_{D}}(n),
\end{equation*}
where
\begin{equation*}
d_{\chi}(n):=\sum_{k|n}\chi(k).
\end{equation*}
Note that, for $\sigma>1$, 
\begin{align*}
\zeta(s)L(s, \chi_{D})=\sum_{n=1}^{\infty}\frac{1}{n^s}\sum_{n=1}^{\infty}\frac{\chi_{D}(n)}{n^s}=
\sum_{n=1}^{\infty}\frac{d_{\chi_{D}}(n)}{n^s}.
\end{align*}

Define
\begin{align*}
\mathscr{D}_{D, \chi}(n):=\sum_{k\mid n}d_{\chi_{D}}(k)\chi(n/k), \quad \text{and}\quad
\mathscr{D}_{D}(n):=\sum_{k\mid n}d_{\chi_{D}}(k).
\end{align*}
Also,
\begin{align*}
\mathscr{D}_{\overline{D}, \overline{\chi}}(n):=\sum_{k\mid n}d_{\overline{\chi}_{D}}(k)\overline{\chi}(n/k), \quad \text{and}\quad
\mathscr{D}_{\overline{D}}(n):=\sum_{k\mid n}d_{\overline{\chi}_{D}}(k).
\end{align*}
Then,
\begin{align*}
\zeta(s)L(s, \chi_{D})L(s, \chi)=\sum_{n=1}^{\infty}\frac{d_{\chi_{D}}(n)}{n^s}\sum_{n=1}^{\infty}\frac{\chi(n)}{n^s}
=\sum_{n=1}^{\infty}\frac{\mathscr{D}_{D, \chi}(n)}{n^s},
\end{align*}
and
\begin{align*}
\zeta^2(s)L(s, \chi_{D})=\sum_{n=1}^{\infty}\frac{1}{n^s}\sum_{n=1}^{\infty}\frac{d_{\chi_{D}}(n)}{n^s}
=\sum_{n=1}^{\infty}\frac{\mathscr{D}_{D}(n)}{n^s}.
\end{align*}

\section{The Case: $\chi_{D}$ even}
\begin{theorem}\label{theorem5.1}
Let $\chi_{D}$ be a nonprincipal even primitive character modulo $|D|$, where $D$ is a
fundamental discriminant.
 Let $q$ be a positive integer, and let $\chi$ be a nonprincipal even primitive character modulo $q.$
Then, for $x>0$ and $\rho>0,$
\begin{align}\label{dDchimain}
&\frac{1}{\Gamma(\rho+1)}
\sum_{n\leq x}\mathscr{D}_{D, \chi}(n)(x^2-n^2)^{\rho} \notag\\
&=\frac{G(\chi)G(\chi_{D})x^{2\rho}}{\pi^{3/2}}\sum_{n=1}^{\infty}\frac{\mathscr{D}_{\overline{D}, \overline{\chi}}(n)}{n}
G_{6,0}^{0,3}\Big(\frac{D^2q^2}{\pi^{6}(nx)^2}\Big)\notag\\
&\qquad -\frac{G(\chi)\sqrt{\pi}L(1, \chi_{D})x^{1+2\rho}}{2q\Gamma(\rho+\tf32)}
\sum_{n=1}^{q-1}\overline{\chi}(n)\log{(2\sin(\pi n/q))}.
\end{align}
\end{theorem}

\begin{proof}
By \eqref{zetafe} and \eqref{evencha}, we see that $\zeta(2s)L(2s, \chi_{D})L(2s, \chi)$ satisfies the functional equation
\begin{align*}
&\pi^{-3s}(|D|q)^s\Gamma^3(s)\zeta(2s)L(2s, \chi_{D})L(2s, \chi) \notag\\
&=\frac{G(\chi)G(\chi_{D})}{\sqrt{|D|q}}\pi^{-3(\tf12-s)}(|D|q)^{(\tf12-s)}\Gamma^3(\tf12-s)\zeta(1-2s)L(1-2s, \chi_{D})L(1-2s, \chi).
\end{align*}
We apply Theorems 2 and 4 in \cite{identitiesI} with $$m=3, ~ r=\frac12,
~ \lambda_n=\mu_n=\dfrac{n^2\pi^3}{|D|q},
~a_n=\mathscr{D}_{D, \chi}(n)  \quad\text{and}\quad  b_n=\dfrac{G(\chi_{D})G(\chi)}{\sqrt{|D|q}}\mathscr{D}_{\overline{D}, \overline{\chi}}(n).$$
Then, we obtain, for $x>0$ and $\rho>0$,
\begin{align}\label{dDchi}
\frac{1}{\Gamma(\rho+1)}
&\sum_{\lambda_n\leq x}\mathscr{D}_{D, \chi}(n)(x-\lambda_n)^{\rho} \notag\\
&=\frac{G(\chi_{D})G(\chi)}{2^{2\rho}\sqrt{|D|q}}
\sum_{n=1}^{\infty}\mathscr{D}_{\overline{D}, \overline{\chi}}(n)\Big(\frac{x}{\mu_n}\Big)^{\frac14+\frac{\rho}{2}}
K_{\frac12+\rho}\big(2^{3}\sqrt{\mu_{n}x};-\tf12;3\big)+Q_{\rho}(x),
\end{align}
where
\begin{equation*}
Q_{\rho}(x)=\df{1}{2\pi{i}}\int_{C_{\rho}}(|D|q)^{s}\df{\Gamma(s)\zeta(2s)L(2s,\chi_{D})L(2s,\chi)x^{s+\rho}}{\pi^{3s}\Gamma(s+\rho+1)}ds,
\end{equation*}
where $C_{\rho}$ is a positively oriented closed curve encircling the poles of the integrand.

The integrand of $Q_{\rho}(x)$ has only a simple pole at $s=\tf12$.  Thus,
\begin{align}\label{Q3}
Q_{\rho}(x)&=\df{(|D|q\pi)^{1/2}L(1,\chi_{D})L(1,\chi)x^{1/2+\rho}}{2\pi^{3/2}\Gamma(\rho+\tfrac32)} \notag\\
&=-\df{(|D|q\pi)^{1/2}L(1,\chi_{D})G(\chi)x^{1/2+\rho}}{2q\pi^{3/2}\Gamma(\rho+\tfrac32)}\sum_{n=1}^{q-1}\overline{\chi}(n)\log{(2\sin(\pi n/q))}.
\end{align}
Now, we replace $x$ by $\tf{\pi^3x^2}{|D|q}$ in \eqref{dDchi} and then apply \eqref{mainmain} and \eqref{Q3} to derive
\begin{align*}
&\frac{1}{\Gamma(\rho+1)}
\sum_{n\leq x}\mathscr{D}_{D, \chi}(n)(x^2-n^2)^{\rho} \notag\\
&=\frac{G(\chi)G(\chi_{D})x^{2\rho}}{\pi^{3/2}}\sum_{n=1}^{\infty}\frac{\mathscr{D}_{\overline{D}, \overline{\chi}}(n)}{n}
G_{6,0}^{0,3}\Big(\frac{D^2q^2}{\pi^{6}(nx)^2}\Big)\notag\\
&\qquad -\frac{G(\chi)\sqrt{\pi}L(1, \chi_{D})x^{1+2\rho}}{2q\Gamma(\rho+\tf32)}
\sum_{n=1}^{q-1}\overline{\chi}(n)\log{(2\sin(\pi n/q))}.
\end{align*}
This completes the proof.
\end{proof}

\begin{theorem}\label{theorem5.2}
Let $\chi_{D}$ be a nonprincipal even primitive character modulo $|D|$, where $D$ is a
fundamental discriminant.
For $x>0$ and $\rho>0,$
\begin{align}\label{dDmain}
&\frac{1}{\Gamma(\rho+1)}
\sum_{n\leq x}\mathscr{D}_{D}(n)(x^2-n^2)^{\rho} \notag\\
=&\frac{G(\chi_{D})x^{2\rho}}{\pi^{3/2}}\sum_{n=1}^{\infty}\frac{\mathscr{D}_{\overline{D}}(n)}{n}
G_{6, 0}^{0, 3}\Big(\frac{D^2}{\pi^{6}(nx)^2}\Big)\notag\\
&+\frac{\sqrt{\pi}x^{1+2\rho}L(1, \chi_{D})}{4\Gamma(\rho+\tf32)}\left\{\frac{2L'(1, \chi_{D})}{L(1, \chi_{D})}
+\frac{\Gamma'(\tf12)}{\sqrt{\pi}}-\psi(\rho+\tf32)
+2\log x+4\gamma\right\},
\end{align}
where $\psi(s)=\dfrac{\Gamma'(s)}{\Gamma(s)}.$
\end{theorem}

\begin{proof}
Similarly, from \eqref{zetafe} and \eqref{evencha}, it is easy to see that $\zeta^2(2s)L(2s, \chi_{D})$ satisfies the functional equation
\begin{align*}
\pi^{-3s}|D|^s\Gamma^3(s)\zeta^2(2s)L(2s, \chi_{D})
=\pi^{-3(\tf12-s)}|D|^{(\tf12-s)}\Gamma^3(\tf12-s)\zeta^2(1-2s)L(1-2s, \overline{\chi}_{|D|}).
\end{align*}
Using Theorems 2 and 4 in \cite{identitiesI} with $$m=3, ~ r=\frac12,
~ \lambda_n=\mu_n=\dfrac{n^2\pi^3}{|D|},
~a_n=\mathscr{D}_{D}(n), \, \quad\text{and}\quad \, b_n=\dfrac{G(\chi_{D})}{\sqrt{|D|}}\mathscr{D}_{\overline{D}}(n),$$
we find that, for $x>0$ and $\rho>0,$
\begin{align}\label{dD}
\frac{1}{\Gamma(\rho+1)}
&\sum_{\lambda_n\leq x}\mathscr{D}_{D}(n)(x-\lambda_n)^{\rho} \notag\\
&=\frac{G(\chi_{D})}{2^{2\rho}\sqrt{|D|}}
\sum_{n=1}^{\infty}\mathscr{D}_{\overline{D}}(n)\Big(\frac{x}{\mu_n}\Big)^{\frac14+\frac{\rho}{2}}
K_{\frac12+\rho}\big(2^{3}\sqrt{\mu_{n}x};-\tf12;3\big)+Q_{\rho}(x),
\end{align}
where
\begin{align*}
Q_{\rho}(x)=\df{1}{2\pi{i}}\int_{C_{\rho}}\df{|D|^{s}\Gamma(s)\zeta^2(2s)L(2s,\chi_{D})x^{s+\rho}}{\pi^{3s}\Gamma(s+\rho+1)}ds.
\end{align*}
where $C_{\rho}$ is a closed curve containing all of the integrand's singularities in its interior. 
Note that $Q_{\rho}(x)$ has only a double pole at $s=\tf12.$ From \eqref{zle}, we have
\begin{align}\label{z2le}
\zeta^2(2s)=\frac{\tf14}{(s-\tf12)^2}+\frac{\gamma}{s-\tf12}+\cdots.
\end{align}
Also, we consider
\begin{align}
\Big(\frac{|D|x}{\pi^{3}}\Big)^{s}&=\Big(\frac{|D|x}{\pi^{3}}\Big)^{1/2}+
\Big(\frac{|D|x}{\pi^{3}}\Big)^{1/2}\log\Big(\frac{|D|x}{\pi^{3}}\Big)(s-\tf12)+\cdots \label{sle},\\
L(2s, \chi_D)&=L(1, \chi_D)+2L'(1, \chi_D)(s-\tf12)+\cdots. \label{Lle}
\end{align}
Using the Laurent expansions \eqref{Gle}, \eqref{1Gle}, \eqref{z2le}, \eqref{sle} and \eqref{Lle}, we can evaluate $Q_{\rho}(x)$ as follows:
\begin{align}\label{Q4}
Q_{\rho}(x)&=\frac{L(1, \chi_D)\sqrt{|D|}x^{\rho+\tf12}}{4\Gamma(\rho+\tf32)\pi}\Big(\frac{2L'(1, \chi_D)}{L(1, \chi_D)}+\frac{\Gamma'(\tf12)}{\sqrt{\pi}}
-\psi(\rho+\tf32)+\log\Big(\df{|D|x}{\pi^3}\Big)+4\gamma\Big).
\end{align}
We replace $x$ by $\tf{\pi^3x^2}{|D|}$ in \eqref{dD} and use \eqref{Q4} and \eqref{mainmain} to deduce that 
\begin{align*}
&\frac{1}{\Gamma(\rho+1)}
\sum_{n\leq x}\mathscr{D}_{D}(n)(x^2-n^2)^{\rho} \notag\\
=&\frac{G(\chi_{D})x^{2\rho}}{\pi^{3/2}}\sum_{n=1}^{\infty}\frac{\mathscr{D}_{\overline{D}}(n)}{n}
G_{6, 0}^{0, 3}\Big(\frac{D^2}{\pi^{6}(nx)^2}\Big)\notag\\
&+\frac{\sqrt{\pi}x^{1+2\rho}L(1, \chi_{D})}{4\Gamma(\rho+\tf32)}\left\{\frac{2L'(1, \chi_{D})}{L(1, \chi_{D})}
+\frac{\Gamma'(\tf12)}{\sqrt{\pi}}-\psi(\rho+\tf32)
+2\log x+4\gamma\right\}.
\end{align*}
This completes our proof.
\end{proof}

\begin{theorem}\label{2ndmain}
Let $\chi_{D}$ be a nonprincipal even primitive character modulo $|D|$, where $D$ is a
fundamental discriminant. For $x>0,$ $0<\theta<1$ and $\rho>0,$
\begin{align}\label{main2}
&\sum_{n\leq x}(x^2-n^2)^{\rho}\sum_{r\mid n}d_{\chi_{D}}\Big(\frac{n}{r}\Big)\cos(2\pi r\theta) \notag\\
&=\frac{\Gamma(\rho+1)G(\chi_D)x^{2\rho}}{2\pi^{3/2}}\sum_{m=1}^{\infty}\frac{d_{\overline{\chi}_{D}}(m)}{m}
\sum_{n=0}^{\infty}\left\{\frac{G_{6, 0}^{0, 3}\Big(\frac{D^2}{\pi^{6}(m(n+\theta)x)^2}\Big)}{n+\theta}
+\frac{G_{6, 0}^{0, 3}\Big(\frac{D^2}{\pi^{6}(m(n+1-\theta)x)^2}\Big)}{n+1-\theta}\right\}\notag\\
&\qquad \qquad -\frac{\Gamma(\rho+1)\sqrt{\pi}L(1, \chi_D)x^{1+2\rho}}{2\Gamma(\rho+\tf32)}\log{(2\sin(\pi\theta))}.
\end{align}
\end{theorem}

\begin{proof}
Similarly to the proof of \eqref{mainmain33}, we prove \eqref{main2} for $\theta=h/q,$ where $q$ is prime and $0<h<q.$
With the same argument, we can derive the following identity, analogous to \eqref{fKcos2}:
\begin{align}\label{dD2}
&\sum_{n\leq x}(x^2-n^2)^{\rho}\sum_{r\mid n}d_{\chi_{D}}\Big(\frac{n}{r}\Big)\cos\Big(\frac{2\pi rh}{q}\Big) \notag\\
&=\frac{q^{2\rho+1}}{\phi(q)}\sum_{n\leq x/q}\left(\Big(\frac{x}{q}\Big)^2-n^2\right)^{\rho}\mathscr{D}_{D}(n)
-\frac{1}{\phi(q)}\sum_{n\leq x}(x^2-n^2)^{\rho}\mathscr{D}_{D}(n)\notag\\
&\qquad+\frac{1}{\phi(q)}\sum_{\substack{\chi\neq\chi_0 \\ \text{even}}}\chi(h)G(\overline{\chi})
\sum_{n\leq x}(x^2-n^2)^{\rho}\mathscr{D}_{D, \chi}(n).
\end{align}
We use the identities \eqref{dDchimain} and \eqref{dDmain} in \eqref{dD2} to deduce that 
\begin{align}\label{dD3}
&\sum_{n\leq x}(x^2-n^2)^{\rho}\sum_{r\mid n}d_{\chi_{D}}\Big(\frac{n}{r}\Big)\cos\Big(\frac{2\pi rh}{q}\Big)  \notag\\
=&\frac{\Gamma(\rho+1)qG(\chi_{D})x^{2\rho}}{\phi(q)\pi^{3/2}}\sum_{n=1}^{\infty}\frac{\mathscr{D}_{\overline{D}}(n)}{n}
G_{6, 0}^{0, 3}\Big(\frac{(Dq)^2}{\pi^{6}(nx)^2}\Big)\notag\\
&+\frac{\Gamma(\rho+1)\sqrt{\pi}x^{1+2\rho}L(1, \chi_{D})}{4\phi(q)\Gamma(\rho+\tf32)}\left\{\frac{2L'(1, \chi_{D})}{L(1, \chi_{D})}
+\frac{\Gamma'(\tf12)}{\sqrt{\pi}}-\psi(\rho+\tf32)
+2\log (\tf{x}{q})+4\gamma\right\} \notag\\
&-\frac{\Gamma(\rho+1)G(\chi_{D})x^{2\rho}}{\phi(q)\pi^{3/2}}\sum_{n=1}^{\infty}\frac{\mathscr{D}_{\overline{D}}(n)}{n}
G_{6, 0}^{0, 3}\Big(\frac{D^2}{\pi^{6}(nx)^2}\Big)\notag\\
&-\frac{\Gamma(\rho+1)\sqrt{\pi}x^{1+2\rho}L(1, \chi_{D})}{4\phi(q)\Gamma(\rho+\tf32)}\left\{\frac{2L'(1, \chi_{D})}{L(1, \chi_{D})}
+\frac{\Gamma'(\tf12)}{\sqrt{\pi}}-\psi(\rho+\tf32)
+2\log x+4\gamma\right\} \notag\\
&+\frac{\Gamma(\rho+1)}{\phi(q)}\sum_{\substack{\chi\neq\chi_0 \\ \text{even}}}\chi(h)G(\overline{\chi})\left\{
\frac{G(\chi)G(\chi_{D})x^{2\rho}}{\pi^{3/2}}\sum_{n=1}^{\infty}\frac{\mathscr{D}_{\overline{D}, \overline{\chi}}(n)}{n}
G_{6,0}^{0,3}\Big(\frac{D^2q^2}{\pi^{6}(nx)^2}\Big) \right.\notag\\
&\qquad \left.-\frac{G(\chi)\sqrt{\pi}L(1, \chi_{D})x^{1+2\rho}}{2q\Gamma(\rho+\tf32)}
\sum_{n=1}^{q-1}\overline{\chi}(n)\log{(2\sin(\pi n/q))} \right\}\notag\\
=&\frac{\Gamma(\rho+1)qG(\chi_{D})x^{2\rho}}{\phi(q)\pi^{3/2}}\sum_{n=1}^{\infty}\frac{\mathscr{D}_{\overline{D}}(n)}{n}
G_{6, 0}^{0, 3}\Big(\frac{(Dq)^2}{\pi^{6}(nx)^2}\Big) \notag\\
&-\frac{\Gamma(\rho+1)G(\chi_{D})x^{2\rho}}{\phi(q)\pi^{3/2}}\sum_{n=1}^{\infty}\frac{\mathscr{D}_{\overline{D}}(n)}{n}
G_{6, 0}^{0, 3}\Big(\frac{D^2}{\pi^{6}(nx)^2}\Big)
-\frac{\Gamma(\rho+1)\sqrt{\pi}x^{1+2\rho}L(1, \chi_{D})}{2\phi(q)\Gamma(\rho+\tf32)}\log q  \notag\\
&+\frac{\Gamma(\rho+1)}{\phi(q)}\sum_{\substack{\chi\neq\chi_0 \\ \text{even}}}\chi(h)\left\{
\frac{qG(\chi_{D})x^{2\rho}}{\pi^{3/2}}\sum_{n=1}^{\infty}\frac{\mathscr{D}_{\overline{D}, \overline{\chi}}(n)}{n}
G_{6,0}^{0,3}\Big(\frac{D^2q^2}{\pi^{6}(nx)^2}\Big) \right.\notag\\
&\qquad \left.-\frac{\sqrt{\pi}L(1, \chi_{D})x^{1+2\rho}}{2\Gamma(\rho+\tf32)}
\sum_{n=1}^{q-1}\overline{\chi}(n)\log{(2\sin(\pi n/q))} \right\}.
\end{align}

Using \eqref{sin} and the easy identity
$$\mathscr{D}_{\overline{D}, \overline{\chi_0}}(n)=\mathscr{D}_{\overline{D}}(n)-\mathscr{D}_{\overline{D}}\big(\tf{n}{q}\big),$$
we can rewrite \eqref{dD3} as
\begin{align*}
&\sum_{n\leq x}(x^2-n^2)^{\rho}\sum_{r\mid n}d_{\chi_{D}}\Big(\frac{n}{r}\Big)\cos\Big(\frac{2\pi rh}{q}\Big) \notag\\
=&\frac{\Gamma(\rho+1)}{\phi(q)}\sum_{\substack{\chi \\ \text{even}}}\chi(h)\left\{
\frac{qG(\chi_{D})x^{2\rho}}{\pi^{3/2}}\sum_{n=1}^{\infty}\frac{\mathscr{D}_{\overline{D}, \overline{\chi}}(n)}{n}
G_{6,0}^{0,3}\Big(\frac{D^2q^2}{\pi^{6}(nx)^2}\Big) \right.\notag\\
&\qquad \left.-\frac{\sqrt{\pi}L(1, \chi_{D})x^{1+2\rho}}{2\Gamma(\rho+\tf32)}
\sum_{n=1}^{q-1}\overline{\chi}(n)\log{(2\sin(\pi n/q))} \right\} \notag\\
=&\frac{\Gamma(\rho+1)qG(\chi_{D})x^{2\rho}}{\phi(q)\pi^{3/2}}
\sum_{n=1}^{\infty}\frac{1}{n}\sum_{k |n}d_{\overline{\chi}_D}(\tf{n}{k})\sum_{\substack{\chi \\ \text{even}}}\chi(h)\overline{\chi}(k)
G_{6,0}^{0,3}\Big(\frac{D^2q^2}{\pi^{6}(nx)^2}\Big) \notag\\
&-\frac{\Gamma(\rho+1)\sqrt{\pi}L(1, \chi_{D})x^{1+2\rho}}{2\phi(q)\Gamma(\rho+\tf32)}
\sum_{n=1}^{q-1}\log{(2\sin(\pi n/q))}\sum_{\substack{\chi \\ \text{even}}}\chi(h)\overline{\chi}(n) \notag\\
=&\frac{\Gamma(\rho+1)G(\chi_D)x^{2\rho}}{2\pi^{3/2}}\sum_{m=1}^{\infty}\frac{d_{\overline{\chi}_{D}}(m)}{m}
\sum_{n=0}^{\infty}\left\{\frac{G_{6, 0}^{0, 3}\Big(\frac{D^2}{\pi^{6}(m(n+h/q)x)^2}\Big)}{n+h/q}
+\frac{G_{6, 0}^{0, 3}\Big(\frac{D^2}{\pi^{6}(m(n+1-h/q)x)^2}\Big)}{n+1-h/q}\right\}\notag\\
&\qquad \qquad -\frac{\Gamma(\rho+1)\sqrt{\pi}L(1, \chi_D)x^{1+2\rho}}{2\Gamma(\rho+\tf32)}\log{(2\sin(\pi h/q))}.
\end{align*}
This completes the proof of \eqref{main2}.
\end{proof}

Multiplying both sides of \eqref{main2} by $w_D$ and using \eqref{R}, we derive the following corollary.
\begin{corollary} Let $x>0$, $0<\theta<1$, and $\rho>0$. 
If $D$ is a fundamental discriminant and $\chi_D$ is a nonprincipal even primitive character,
then
\begin{align*}
&\sum_{n\leq x}(x^2-n^2)^{\rho}\sum_{r\mid n}R_D\Big(\frac{n}{r}\Big)\cos(2\pi r\theta) \notag\\
&=\frac{\Gamma(\rho+1)G(\chi_D)x^{2\rho}}{2\pi^{3/2}}\sum_{m=1}^{\infty}\frac{R_D(m)}{m}
\sum_{n=0}^{\infty}\left\{\frac{G_{6, 0}^{0, 3}\Big(\frac{D^2}{\pi^{6}(m(n+\theta)x)^2}\Big)}{n+\theta}
+\frac{G_{6, 0}^{0, 3}\Big(\frac{D^2}{\pi^{6}(m(n+1-\theta)x)^2}\Big)}{n+1-\theta}\right\}\notag\\
&\qquad \qquad -\frac{\Gamma(\rho+1)w_D\sqrt{\pi}L(1, \chi_D)x^{1+2\rho}}{2\Gamma(\rho+\tf32)}\log{(2\sin(\pi\theta))}.
\end{align*}
\end{corollary}

\section{Big-O results}

We employ a theorem of Chandrasekharan and Narasimhan \cite{annals2}. We first provide the general setting \cite[p. 95--98]{annals2}.

 \begin{definition} Let $a(n)$ and $b(n)$ be two sequences of complex numbers, where not all terms are equal to 0 in either sequence. Let $\lambda_n$ and $\mu_n$ be two sequences of positive numbers, strictly increasing to $\infty$. Let $\delta>0$. Throughout, $s=\sigma +it$, where $\sigma$ and $t$ are both real. For $N\geq1$, let
 \begin{equation}\label{Delta}
 \Delta(s):=\prod_{n=1}^{N}\Gamma(\alpha_n s+\beta_n),
 \end{equation}
 where, for $1\leq n\leq N$, $\beta_n$ is complex and $\alpha_n>0$.  Assume that
 \begin{equation}\label{A}
 A:=\sum_{n=1}^N\alpha_n\geq 1.
 \end{equation}
 Let
 \begin{equation*}
 \varphi(s):=\sum_{n=1}^{\infty}\df{a(n)}{\lambda_n^s}\qquad \text{and}\qquad \psi(s):=\sum_{n=1}^{\infty}\df{b(n)}{\mu_n^s}
 \end{equation*}
  converge absolutely in some half-plane, and suppose they satisfy the functional equation
 \begin{equation*}
 \Delta(s)\varphi(s)=\Delta(\delta-s)\psi(\delta-s).
 \end{equation*}
 Furthermore, assume that
   there exists in the $s$-plane a domain $\mathfrak{D}$, which is the exterior of a compact set $S$, in which there exists an analytic function $\chi$ with the properties
 \begin{equation*}
 \lim_{|t|\to\infty}\chi(s)=0,
 \end{equation*}
 uniformly in every interval $-\infty<\sigma_1\leq \sigma\leq\sigma_2<\infty$, and
 \begin{align*}
 \chi(s)&=\Delta(s)\varphi(s), \qquad \sigma > \alpha,\\
 \chi(s)&=\Delta(\delta-s)\psi(\delta-s), \qquad \sigma<\beta,
 \end{align*}
 where $\alpha$ and $\beta$ are particular constants.
 \end{definition}

For $\rho\geq0$, let
\begin{equation*}
A_{\rho}(x):=\df{1}{\Gamma(\rho+1)}\sum_{\lambda_n\leq x}a(n)(x-\lambda_n)^{\rho},
\end{equation*}
where the prime $\prime$ indicates that if $x=\lambda_n$ and $\rho=0$, the last term is to be multiplied by $\tf12$. Furthermore, let
 \begin{equation*}\label{poles}
 Q_{\rho}(x):=\df{1}{2\pi i}\int_{\mathcal{C_{\rho}}}\df{\Gamma(s)\varphi(s)}{\Gamma(s+\rho+1)}x^{s+\rho}ds,
 \end{equation*}
 where $\mathcal{C_{\rho}}$ is a closed curve enclosing all of the singularities of the integrand to the right of $\sigma=-\rho-1-k$, where $k$ is chosen such that $k>|\delta/2-1/(4A)|$, and all of the singularities of $\varphi(s)$ lie in $\sigma>-k$.

\begin{theorem}\cite[p. 98--99, Theorem 3.2]{annals2}\label{Othm}
Suppose that the functional equation
 $$\Delta(s)\varphi(s)=\Delta(\delta-s)\psi(\delta-s)$$
  is satisfied. If $\rho\geq 2A\beta-A\delta-\frac12$, where $\beta$ is such that 
  $\sum_{n=1}^{\infty}|b_n|\mu_n^{-\beta}<\infty,$ then 
  \begin{equation}\label{A-Q}
  A_{\rho}(x)-Q_{\rho}(x)=O(x^{\theta}),
  \end{equation}
  where 
  \begin{equation}\label{theta}
  \theta=\frac{A\delta+\rho(2A-1)-1/2}{2A}.
  \end{equation}
\end{theorem}

\begin{theorem}\label{theorem6.3} Let $q$ be prime and $0<h<q.$ For $\rho> r_1/2,$
\begin{align*}
&\sum_{n\leq x}(x^2-n^2)^{\rho}\sum_{r\mid n}f_K\Big(\frac{n}{r}\Big)\cos\Big(\frac{2\pi rh}{q}\Big) \notag\\
&=-\frac{\sqrt{\pi}\Gamma(\rho+1)\gamma_{-1}(K)}{2\Gamma(\rho+\tf32)}\log{(2\sin(\pi h/q))}x^{1+2\rho}+
O\Big(x^{2\rho+\tf{1}{2}-\tf{2\rho+1}{2(r_1+1)}}\Big).
\end{align*}
\end{theorem}

\begin{proof} For the convenience of readers, we recall \eqref{fKcos2}:
\begin{align}\label{fKcos2p}
&\sum_{n\leq x}(x^2-n^2)^{\rho}\sum_{r\mid n}f_K\Big(\frac{n}{r}\Big)\cos\Big(\frac{2\pi rh}{q}\Big) \notag\\
=&\frac{q^{2\rho+1}}{\phi(q)}\sum_{n\leq x/q}\left(\Big(\frac{x}{q}\Big)^2-n^2\right)^{\rho}D_K(n)
-\frac{1}{\phi(q)}\sum_{n\leq x}(x^2-n^2)^{\rho}D_K(n) \notag\\
&\qquad+\frac{1}{\phi(q)}\sum_{\substack{\chi\neq\chi_0 \\ \text{even}}}\chi(h)G(\overline{\chi})
\sum_{n\leq x}(x^2-n^2)^{\rho}D_{K, \chi}(n).
\end{align}
We apply Theorem \ref{Othm} to estimate the sums 
\begin{equation}\label{sums}
\sum_{n\leq x}(x^2-n^2)^{\rho}D_K(n) \quad 
\text{and} \quad \sum_{n\leq x}(x^2-n^2)^{\rho}D_{K, \chi}(n).
\end{equation}
We begin with the first sum in \eqref{sums}. 
From \eqref{feq2}, \eqref{Delta} and \eqref{A}, we have $A=r_1+1$ and $\delta=\frac12.$ So, by \eqref{theta},
\begin{equation}\label{theta2}
\theta=\rho+\frac14-\frac{2\rho+1}{4(r_1+1)}.
\end{equation}
Also, note that $\beta=\tf12+\epsilon$, where $\epsilon >0$.  Thus, by \eqref{A-Q}, we see that, for $\rho> r_1/2,$
\begin{align*}
\frac{1}{\Gamma(\rho+1)}\sum_{\lambda_n\leq x}(x-\lambda_n)^{\rho}D_K(n)=Q_{\rho}(x)+O(x^{\theta}),
\end{align*}
where $Q_{\rho}(x)$ is given in \eqref{Q2}. Replacing $x$ by $\tf{\pi^{r_1+1}x^2}{|\Delta_K|},$ as in \eqref{star}, we have
\begin{align}\label{sum1}
\frac{1}{\Gamma(\rho+1)}\sum_{n\leq x}(x^2-n^2)^{\rho}D_K(n)&=\frac{\sqrt{\pi}x^{1+2\rho}}{2\Gamma(\rho+\tf32)}\left\{
\gamma_0(K)+\gamma_{-1}(K)\gamma+\frac{\gamma_{-1}(K)\Gamma'(\tf12)}{2\sqrt{\pi}}\right.\notag \\
&\qquad\left.-\frac12\gamma_{-1}(K)\psi(\rho+\tf32)+\gamma_{-1}(K)\log x\right\}+O(x^{2\theta}).
\end{align}

We estimate the second sum in \eqref{sums} in a similar fashion.  From \eqref{feq},  \eqref{Delta},  \eqref{A} and \eqref{theta},
we see that the values of  $A, \delta$ and $\theta$ are the same as above. Applying \eqref{A-Q} yields for $\rho> r_1/2,$
\begin{equation*}
\frac{1}{\Gamma(\rho+1)}\sum_{\lambda_n\leq x}(x-\lambda_n)^{\rho}D_{K, \chi}(n)=Q_{\rho}(x)+O(x^{\theta}),
\end{equation*}
where $Q_{\rho}(x)$ is evaluated in \eqref{Q}.
We replace $x$ by $\tf{\pi^{r_1+1}x^2}{q|\Delta_K|}$ to obtain 
\begin{align}\label{sum2}
&\frac{1}{\Gamma(\rho+1)}
\sum_{n\leq x}(x^2-n^2)^{\rho}D_{K, \chi}(n) \notag\\
&\qquad =-\frac{G(\chi)\sqrt{\pi}\gamma_{-1}(K)x^{1+2\rho}}{2q\Gamma(\rho+\tf32)}
\sum_{n=1}^{q-1}\overline{\chi}(n)\log{(2\sin(\pi n/q))} +O(x^{2\theta}).
\end{align}
Now, substituting \eqref{sum1} and \eqref{sum2} into \eqref{fKcos2p}, we deduce that 
\begin{align}
&\sum_{n\leq x}(x^2-n^2)^{\rho}\sum_{r\mid n}f_K\Big(\frac{n}{r}\Big)\cos\Big(\frac{2\pi rh}{q}\Big)\notag\\
&=-\frac{\sqrt{\pi}\Gamma(\rho+1)\gamma_{-1}(K)\log q}{2\phi(q)\Gamma(\rho+\tf32)}x^{1+2\rho}+O(x^{2\theta})\notag\\
&\quad -\frac{\sqrt{\pi}\Gamma(\rho+1)\gamma_{-1}(K)}{2\phi(q)\Gamma(\rho+\tf32)}x^{1+2\rho}
\sum_{\substack{\chi\neq\chi_0 \\ \text{even}}}\chi(h)
\sum_{n=1}^{q-1}\overline{\chi}(n)\log{(2\sin(\pi n/q))}\notag\\
&=-\frac{\sqrt{\pi}\Gamma(\rho+1)\gamma_{-1}(K)}{2\phi(q)\Gamma(\rho+\tf32)}x^{1+2\rho}
\sum_{\substack{\chi\\ \text{even}}}\chi(h)
\sum_{n=1}^{q-1}\overline{\chi}(n)\log{(2\sin(\pi n/q))}+O(x^{2\theta})\notag\\
&=-\frac{\sqrt{\pi}\Gamma(\rho+1)\gamma_{-1}(K)}{2\phi(q)\Gamma(\rho+\tf32)}x^{1+2\rho}
\sum_{n=1}^{q-1}\log{(2\sin(\pi n/q))}\sum_{\substack{\chi\\ \text{even}}}\chi(h)\overline{\chi}(n)+O(x^{2\theta})\notag\\
&=-\frac{\sqrt{\pi}\Gamma(\rho+1)\gamma_{-1}(K)}{2\Gamma(\rho+\tf32)}x^{1+2\rho}\log{(2\sin(\pi h/q))}+O(x^{2\theta}),\label{starstar}
\end{align}
where we employed \eqref{sin} and \eqref{chieven}.  Using \eqref{theta2} in \eqref{starstar}, we
 complete our proof. 
\end{proof}

\begin{theorem}\label{theorem6.4}
Let $q$ be prime and $0<h<q.$ For $\rho> 1,$
\begin{align*}
&\sum_{n\leq x}(x^2-n^2)^{\rho}\sum_{r\mid n}d_{\chi_{D}}\Big(\frac{n}{r}\Big)\cos\Big(\frac{2\pi rh}{q}\Big) \\
&=-\frac{\sqrt{\pi}\Gamma(\rho+1)L(1, \chi_{D})}{2\Gamma(\rho+\tf32)}\log{(2\sin(\pi h/q))}x^{1+2\rho}+O(x^{(5\rho+1)/3}).
\end{align*}
\end{theorem}

\begin{proof}
For convenience, we record \eqref{dD2}: 
\begin{align}\label{dD2'}
&\sum_{n\leq x}(x^2-n^2)^{\rho}\sum_{r\mid n}d_{\chi_{D}}\Big(\frac{n}{r}\Big)\cos\Big(\frac{2\pi rh}{q}\Big) \notag\\
&=\frac{q^{2\rho+1}}{\phi(q)}\sum_{n\leq x/q}\left(\Big(\frac{x}{q}\Big)^2-n^2\right)^{\rho}\mathscr{D}_{D}(n)
-\frac{1}{\phi(q)}\sum_{n\leq x}(x^2-n^2)^{\rho}\mathscr{D}_{D}(n)\notag\\
&\qquad+\frac{1}{\phi(q)}\sum_{\substack{\chi\neq\chi_0 \\ \text{even}}}\chi(h)G(\overline{\chi})
\sum_{n\leq x}(x^2-n^2)^{\rho}\mathscr{D}_{D, \chi}(n).
\end{align}

Next, apply Theorem \ref{Othm}.  Observe that $A=3, \delta =\frac12$, $\beta=\frac12+\epsilon$, for $\epsilon >0$, and $\rho>1$. Referring to Theorem \ref{theorem5.2} and its proof, we replace $x$ by $\frac{\pi^3x^2}{|D|q}$.  Thus, we can derive, for $\rho> 1,$
\begin{align}\label{sum3}
&\frac{1}{\Gamma(\rho+1)}
\sum_{n\leq x}\mathscr{D}_{D}(n)(x^2-n^2)^{\rho} \notag\\
&=\frac{\sqrt{\pi}L(1, \chi_{D})}{4\Gamma(\rho+\tf32)}x^{1+2\rho}\left\{\frac{2L'(1, \chi_{D})}{L(1, \chi_{D})}
+\frac{\Gamma'(\tf12)}{\sqrt{\pi}}-\psi(\rho+\tf32)
+2\log x+4\gamma\right\}+O(x^{(5\rho+1)/3}),
\end{align}
where we recall that $\psi(s)=\dfrac{\Gamma'(s)}{\Gamma(s)}.$ 

Next, refer to Theorem \ref{theorem5.1} and its proof, in particular, \eqref{Q3}, in which we replaced $x$ by $\frac{\pi^{r_1+1}x^2}{q|\Delta_K|}$.  Thus, for $\rho> 1,$
\begin{align}
&\frac{1}{\Gamma(\rho+1)}
\sum_{n\leq x}\mathscr{D}_{D, \chi}(n)(x^2-n^2)^{\rho} \notag\\
&=-\frac{G(\chi)\sqrt{\pi}L(1, \chi_{D})}{2q\Gamma(\rho+\tf32)}x^{1+2\rho}
\sum_{n=1}^{q-1}\overline{\chi}(n)\log{(2\sin(\pi n/q))}+O(x^{(5\rho+1)/3}).\label{sum4} 
\end{align}
Now, we substitute \eqref{sum3} and \eqref{sum4} into \eqref{dD2'} to find that 
\begin{align*}
&\frac{1}{\Gamma(\rho+1)}\sum_{n\leq x}(x^2-n^2)^{\rho}\sum_{r\mid n}d_{\chi_{D}}\Big(\frac{n}{r}\Big)\cos\Big(\frac{2\pi rh}{q}\Big)\notag\\
=&-\frac{\sqrt{\pi}L(1, \chi_{D})\log q}{2\phi(q)\Gamma(\rho+\tf32)}x^{1+2\rho}+O(x^{(5\rho+1)/3})\notag\\
&-\frac{\sqrt{\pi}L(1, \chi_{D})x^{1+2\rho}}{2\phi(q)\Gamma(\rho+\tf32)}
\sum_{n=1}^{q-1}\log{(2\sin(\pi n/q))}\sum_{\substack{\chi\neq\chi_0 \\ \text{even}}}\chi(h)\overline{\chi}(n)\notag\\
=&-\frac{\sqrt{\pi}L(1, \chi_{D})x^{1+2\rho}}{2\phi(q)\Gamma(\rho+\tf32)}
\sum_{n=1}^{q-1}\log{(2\sin(\pi n/q))}\sum_{\substack{\chi \\ \text{even}}}\chi(h)\overline{\chi}(n)+O(x^{(5\rho+1)/3})\notag\\
=&-\frac{\sqrt{\pi}L(1, \chi_{D})x^{1+2\rho}}{2\Gamma(\rho+\tf32)}\log{(2\sin(\pi h/q))}+O(x^{(5\rho+1)/3}),  
\end{align*}
where we employed  \eqref{sin} and \eqref{chieven}. 
This completes the proof of the theorem. 
\end{proof}


\section{The Meijer $G$-Function}\label{meijer}

Recall the definition of the Meijer $G$-function in Definition \ref{definition1}.  Also, recall that to complete the proof of Theorem \ref{rsub1}, we need to prove the following theorem.

\begin{theorem}\label{rsub2} If $K_{\rho+1/2}(x;-\tf12;r_1+1)$ is defined by \eqref{mainmain}, then
\begin{gather*}
 K_{\rho+1/2}(x;-\tf12;r_1+1)\label{k1}\\=2^{-\rho+(r_1+1)/2}x^{\rho-1/2} G_{2(r_1+1),0}^{0,r_1+1}\left(\df{4^{r_1+1}}{ x^2}\Bigg|
 \begin{matrix}\tf12,\tf12,\dots, \tf12,\rho+1,0,0,\dots, 0\\-\end{matrix}\right),\notag
  \end{gather*}
  where there are $r_1\,\{0\}$'s and $(r_1+1)\, \{\tf12\}$'s.
  \end{theorem}

\begin{proof}
  For the remainder of the proof, we set
\begin{equation}\label{axu}
a_j=\df{x}{u_{j+1}u_{j+2}\cdots u_{r_1}}, \qquad 1\leq j\leq r_1-1,\qquad a_{r_1}=x.
\end{equation}
 Employing \eqref{axu}, and replacing $u_1$ by $u$, we see that the inner integral is
\begin{equation}\label{1b}
f_1(\rho;a_1):=\sqrt{\df{2}{\pi}}\int_0^{\infty}u^{\rho-1/2}\cos u\,J_{\rho+1/2}(a_1/u)du.
\end{equation}
We use a representation for the ordinary Bessel function $J_{\nu}(x)$ in terms of a $G$-function \cite[p.~380]{erd1}, namely,
\begin{equation}\label{0}
x^{\mu}J_{\nu}(x)=2^{\mu}G^{1,0}_{0,2}\left(\df{x^2}{4}\Bigg|\begin{matrix}-\\\tf12(\mu+\nu),\tf12(\mu-\nu)\end{matrix}\right).
\end{equation}
Applying \eqref{0} with $\mu=-\rho+\tf12$, $\nu=\rho+\tf12$, and $x=a_1/u$, we find that
\begin{equation*}
\left(\df{a_1}{u}\right)^{-\rho+1/2}J_{\rho+1/2}\left(\df{a_1}{u}\right)=2^{-\rho+1/2}
G^{1,0}_{0,2}\left(\df{a_1^2}{4u^2}\Bigg|
 \begin{matrix}-\\\tf12,-\rho \end{matrix}\right),
\end{equation*}
or
\begin{align}
u^{\rho-1/2}J_{\rho+1/2}\left(\df{a_1}{u}\right)&=\left(\df{a_1}{2}\right)^{\rho-1/2}
G^{1,0}_{0,2}\left(\df{a_1^2}{4u^2}\Bigg|
 \begin{matrix}-\\ \tf12,-\rho \end{matrix}\right)\label{2b}\\
 &=\left(\df{a_1}{2}\right)^{\rho-1/2}
G^{0,1}_{2,0}\left(\df{4u^2}{a_1^2}\Bigg|
 \begin{matrix}\tf12,\rho+1\\- \end{matrix}\right).\label{2a}
\end{align}
 Note that in \eqref{2b}, we employed the definition of $G$ in \eqref{G} to switch the roles of the parameters $\tf12$ and $-\rho$ in \eqref{2b} and to obtain an alternative set of  parameters in \eqref{2a}.
Hence, by \eqref{1b} and \eqref{2a},
\begin{equation}\label{3bb}
f_1(\rho;a_1)=\left(\df{a_1}{2}\right)^{\rho-1/2}\sqrt{\df{2}{\pi}}\int_0^{\infty}\cos u\,
G_{2,0}^{0,1}\left(\df{4u^2}{a_1^2}\Bigg|
 \begin{matrix}\tf12,\rho+1\\- \\\end{matrix}\right)du.
 \end{equation}

We now state the general cosine transform of $G$ \cite[p.~420, Formula (8)]{erd2}. If $a>0$ and $c\neq 0$,
\begin{gather}\label{4b}
\int_0^{\infty}\cos(cx)G_{p,q}^{m,n}\left(ax^2\Bigg| \begin{matrix}a_1,a_2,\dots,a_p\\b_1,b_2,\dots, b_q\end{matrix}\right)dx
=\df{\sqrt{\pi}}{c}G_{p+2,q}^{m,n+1}\left(\df{4a}{c^2}\Bigg| \begin{matrix}\tf12,a_1,a_2,\dots,a_p,0\\b_1,b_2,\dots, b_q\end{matrix}\right),
\end{gather}
provided that $p+q\leq 2(m+n)$;  $0\leq m+n-\tf12 p -\tf12 q$; $c>0$, $\Realp\,a_j\leq\tf12, j=1,2,\dots, n$; and $\Realp\,b_j\geq-\tf12, j=1,2,\dots,m$.
  Utilizing \eqref{4b} in \eqref{3bb}, we find that
  \begin{align}
f_1(\rho;a_1)&=2^{1/2}\left(\df{a_1}{2}\right)^{\rho-1/2}G_{4,0}^{0,2}\left(\df{16}{a_1^2}\Bigg|
 \begin{matrix}\tf12,\tf12, \rho+1,0\\-\end{matrix}\right)\label{5b}\\
&=2^{-\rho+1}\df{x^{\rho-1/2}}{(u_2\cdots u_{r_1})^{\rho-1/2}}G_{4,0}^{0,2}\left(\df{16u_2^2}{a_2^2}\Bigg|
 \begin{matrix}\tf12,\tf12, \rho+1,0\\-\end{matrix}\right),\label{6b}
 \end{align}
 where in the last step we used \eqref{axu}.

 Next, we put \eqref{6b} in \eqref{mainmain} and repeat the process.  Note that with each application we use \eqref{axu}, and increase the number of $\tf12$'s and the number of $0$'s, each by 1.  Hence,
 \begin{align*}
 &K_{\rho+1/2}(x;-\tf12;r_1+1)\\
 =&2^{-\rho+1}x^{\rho-1/2}\int_0^{\infty}\sqrt{u_{r_1}}J_{-1/2}(u_{r_1})du_{r_1}\cdots\int_0^{\infty}\sqrt{u_2}J_{-1/2}(u_{2})du_2
 )\\
 &\times G_{4,0}^{0,2}\left(\df{4^2 u_2^2}{a_2^2}\Bigg|
 \begin{matrix}\tf12, \tf12,\rho+1,0\\-\end{matrix}\right)\\
 =&2^{-\rho+3/2}x^{\rho-1/2}
 \int_0^{\infty}\sqrt{u_{r_1}}J_{-1/2}(u_{r_1})du_{r_1}\cdots\int_0^{\infty}\sqrt{u_3}J_{-1/2}(u_{3})du_3\\
 &\times G_{6,0}^{0,3}\left(\df{4^3 u_3^2}{a_3^2}\Bigg|
 \begin{matrix}\tf12,\tf12,\tf12,\rho+1,0,0 \\-\end{matrix}\right)\\
 =&\cdots\cdots\cdots\\
   =&2^{-\rho+(r_1+1)/2}x^{\rho-1/2} G_{2(r_1+1),0}^{0,r_1+1}\left(\df{4^{r_1+1}}{ x^2}\Bigg|
 \begin{matrix}\tf12,\tf12,\dots, \tf12,\rho+1,0,0,\dots, 0\\-\end{matrix}\right),
  \end{align*}
  where there are $r_1\,\{0\}$'s and $(r_1+1)\, \{\tf12\}$'s.
In the last step, the integral evaluation is given by
  $$\int_0^{\infty}\sqrt{u}J_{-1/2}(u)G_{2n,0}^{0,n}\left(\df{u^2}{a^2}\Bigg|
 \begin{matrix}\{\tf12\},\rho+1,\{0\}\\-\end{matrix}\right)du
 =\sqrt{2}G_{2n+2,0}^{0,n+1}\left(\df{4}{a^2}\Bigg|
 \begin{matrix}\{\tf12\},\tf12,\rho+1,\{0\},0\\-\end{matrix}\right),$$
where there are $r_1\, \{0\}$'s and $(r_1+1)\,\{\tf12\}$'s.
This completes the proof of \eqref{k1}.
\end{proof}

\section{Special Case}
We consider the special case of \eqref{K} when $\rho=0$ and $r_1=1$.
Recall that \cite[pp.~54--55]{watson}
\begin{equation*}\label{1/2}
J_{-1/2}(z)=\sqrt{\df{2}{\pi z}}\cos z \qquad \text{and}\qquad J_{1/2}(z)=\sqrt{\df{2}{\pi z}}\sin z.
\end{equation*}
Thus, \eqref{K} reduces to, with $x=y^2$
\cite[p.~74, Equation (3.5)]{bessel2bbskaz},
\begin{align}\label{specialcase}
K_{\frac12}(x;-\tf12;2)=&\int_0^{\infty}J_{-1/2}(u)J_{1/2}\left(\dfrac{x}{u}\right)du\notag\\
=&\left(\df{2}{\pi\,y}\right)\int_0^{\infty}\cos u\sin\left(\df{y^2}{u}\right)du\notag\\
=&\left(\df{2}{\pi\,y}\right)\left(-y\left(\df{\pi}{2}Y_1(2y)+K_1(2y)\right)\right)
\notag\\
=&-Y_1(2y)-\df{2}{\pi}K_1(2y),
\end{align}
where we used \cite[p.~74, Equation (3.5)]{bessel2bbskaz} (or \cite[p.~184, Formula (3)]{watson}), and
where $Y_1$ and $K_1$ are the Bessel functions defined in \eqref{b.Y} and \eqref{b.K}, respectively. 
\section{Acknowledgement}
The authors are very grateful to Larry Glasser for guiding them to a proof of 
Theorem \ref{rsub2}.

\end{document}